\documentclass{amsart}

\usepackage{amsmath, amsthm, amscd, amsfonts, amssymb, enumerate, verbatim, newlfont, calc, graphicx, color}
\usepackage[bookmarksnumbered, colorlinks, plainpages]{hyperref}
\usepackage{tikz} 

\newtheorem{theorem}{Theorem}[section]

\newtheorem{lemma}[theorem]{Lemma}
\newtheorem{proposition}[theorem]{Proposition}
\newtheorem{corollary}[theorem]{Corollary}
\theoremstyle{definition}
\newtheorem{definition}[theorem]{Definition}

\newtheorem{remark}[theorem]{Remark}

\newtheorem{example}[theorem]{Example}

\usepackage{pstricks}
\usepackage{epsfig}
\usepackage{pst-grad}
\usepackage{pst-plot}

\begin{document}
\author[M. Nasernejad,  A. A. Qureshi, K. Khashyarmanesh, and  L. G. Roberts]{ Mehrdad ~Nasernejad$^{1}$,  Ayesha Asloob Qureshi$^{2,*}$, Kazem Khashyarmanesh$^{1}$, and Leslie G. Roberts$^{3}$}
\title[Classes of normally and nearly normally torsion-free ideals]{Classes of normally and nearly normally torsion-free monomial ideals}
\subjclass[2010]{13B25, 13F20, 05E40.} 
\keywords { Normally torsion-free ideals, Nearly normally torsion-free ideals, Associated prime ideals, $t$-spread monomial ideals, Hypergraphs.}
\thanks{$^*$Corresponding author}

\thanks{E-mail addresses:   m\_{nasernejad@yahoo.com},  aqureshi@sabanciuniv.edu, khashyar@ipm.ir, and  robertsl@queensu.ca}  
\maketitle
 
\begin{center}
{\it
$^1$Department of Pure Mathematics,
 Ferdowsi University of Mashhad,\\
P.O.Box 1159-91775, Mashhad, Iran\\
 $^2$Sabanc\i\;University, Faculty of Engineering and Natural Sciences, \\
Orta Mahalle, Tuzla 34956, Istanbul, Turkey\\
$^3$Department of Mathematics and Statistics, 
Queen's University, \\
Kingston, Ontario, Canada, K7L 3N6
}
\end{center}

\vspace{0.4cm}

\begin{abstract}
In this paper, our main focus is to explore different classes of nearly normally torsion-free ideals. We first characterize all finite simple connected graphs with nearly normally torsion-free cover ideals. Next, we characterize all normally torsion-free $t$-spread principal Borel ideals that can also be viewed as edge ideals of uniform multipartite hypergraphs. 
\end{abstract}
\vspace{0.4cm}

\section{Introduction}

Let $\mathcal{H}=(V_{\mathcal{H}},E_{\mathcal{H}})$ be a simple hypergraph on vertex set $V_{\mathcal{H}}$ and edge set $E_{\mathcal{H}}$.  The edge ideal of $\mathcal{H}$, denoted by $I(\mathcal{H})$, is the ideal generated by the monomials corresponding to the edges of $\mathcal{H}$. A hypergraph $\mathcal{H}$ is called {\em Mengerian} if it satisfies a certain  min-max equation, which is known as the Mengerian property in hypergraph theory or as the max-flow min-cut property in integer programming. Algebraically, it is equivalent to $I(\mathcal{H})$ being normally torsion-free, see \cite[Corollary 10.3.15]{HH1}, \cite[Theorem 14.3.6]{V1}.

Let $R$ be a commutative Noetherian ring and $I$ be  an ideal of $R$. In addition,  let $\mathrm{Ass}_R(R/I)$ be the set of all prime ideals associated to  $I$. An ideal $I$ is called {\it normally torsion-free} if  $\mathrm{Ass}(R/I^k)\subseteq \mathrm{Ass}(R/I)$, for all $k\geq 1$  \cite[Definition 1.4.5]{HH1}. In other words, if $I$ has no embedded primes, then $I$ is normally torsion-free if and only if the ordinary powers of $I$ coincide with the symbolic powers of $I$. Normally torsion-free ideals have been a topic of several papers, however, a few classes of these ideals originate  from graph theory. Simis, Vasconcelos and Villarreal showed in \cite{SVV}  that  a finite simple graph is bipartite if and only if its edge ideal is normally torsion-free. An analogue of bipartite graphs in higher dimensions is considered to be a hypergraph that avoids ``special odd cycles". Such hypergraphs are called {\em balanced}, and a well-known result of Fulkerson, Hoffman and Oppenheim in \cite{FHO} states that balanced hypergraphs are Mengerian. It follows  immediately  that the edge ideals of balanced hypergraphs are normally torsion-free. However, unlike the case of bipartite graphs, it should be noted that the converse of this statement is not true. We refer the reader to \cite{B} for the related definitions in hypergraph theory.

     A monomial ideal $I$ in a polynomial  ring $R=K[x_1, \ldots, x_n]$ over a field $K$ is called {\it nearly normally torsion-free}  if there exist a positive integer $k$ and a monomial prime ideal
 $\mathfrak{p}$ such that $\mathrm{Ass}_R(R/I^m)=\mathrm{Min}(I)$ for all $1\leq m\leq k$, and 
 $\mathrm{Ass}_R(R/I^m) \subseteq \mathrm{Min}(I) \cup \{\mathfrak{p}\}$ for all $m \geq k+1$, see \cite[Definition 2.1]{Claudia}. This concept generalizes normally torsion-freeness to some extent. Recently, the author, in \cite[Theorem 2.3]{Claudia}, characterized all connected graphs whose edge ideals are nearly normally torsion-free. Our main aim is to further explore different classes of nearly normally torsion-free ideals and normally torsion-free ideals that originate from hypergraph theory. 
 
 This paper is organized as follows:  In Section~\ref{prem}, we provide some notation  and definitions which appear throughout the paper. In Section~\ref{generalresults}, we give some results to develop different criteria that help to investigate whether an ideal is normally torsion-free. For this purpose, we employ  the notion of monomial localization of a monomial ideal with respect to a prime monomial ideal, see Lemmas \ref{Lem. 1} and \ref{Lem. 2}. In what follows, we provide  two applications of Corollary \ref{Cor. 1}. In the first application, we give a class of nearly normally torsion-free monomial ideals, which is concerned with the monomial ideals of intersection type (Proposition  \ref{App. 1}). For the second application, we turn our attention to the $t$-spread principal Borel ideals, and reprove one of the main results of \cite{Claudia} (Proposition  \ref{App. 2}). We close Section~\ref{generalresults} by giving a concrete example of a nearly normally torsion-free ideal that does not have the strong persistence property, see Definition~\ref{spersistence}. Note that the ideal given in our example is not square-free. The question that whether nearly normally torsion-free square-free monomial ideals have the persistence or the strong persistence property still remains open. 

In Section~\ref{coverideals}, one of our main results  is Theorem \ref{Main. 1} which proves that if $G$ is a finite simple connected graph, then the cover ideal of $G$ is nearly normally torsion-free if and only if $G$ is either a bipartite graph or an almost bipartite graph.   In proving 
this we found reference  \cite {JS} very helpful.   The second main result of Section~\ref{coverideals} is Corollary~\ref{Cor. NFT1} which provides a new class of normally torsion-free ideals based on the existing ones. To reach this purpose, we start with a hypergraph $\mathcal{H}$ whose edge ideal is normally torsion-free, in other words, we take a Mengerian hypergraph $\mathcal{H}$. By using the notion of coloring of hypergraphs, in Theorem~\ref{NTF1}, we show that the new hypergraph $\mathcal{H'}$ obtained by adding a ``whisker" to $\mathcal{H}$ is also Mengerian. Here, to add a whisker to $\mathcal{H}$, one adds a new vertex and an edge of size two consisting of this new vertex and an existing vertex of $\mathcal{H}$. 

In Section~\ref{borel}, we study $t$-spread principal Borel ideals, see Definition~\ref{boreldef}. An interesting feature of $t$-spread principal Borel ideals is noted in Theorem~\ref{complete} that shows a $t$-spread principal Borel ideal is normally torsion-free if and only if it can viewed as an edge ideal of a certain $d$-uniform $d$-partite hypergraph. As we noted in Example~\ref{oddcycle}, the hypergraph associated to a $t$-spread principal Borel ideal may contain special odd cycles, and hence these hypergraphs are not necessarily balanced. This limits us to use the result of Fulkerson, Hoffman and Oppenheim in \cite{FHO}. We prove Theorem~\ref{complete} by applying a combination of algebraic and combinatorial techniques. We make use of \cite[Theorem 3.7]{SNQ} which gives a criterion to check whether an ideal is normally torsion-free. In addition, we use the notion of linear relation graph associated to monomial ideals, see Definition~\ref{linearrelationgraphdef}, to study the set of associated primes ideals of $t$-spread principal Borel ideals.

\section{Preliminaries}\label{prem}

Let  $R$ be  a commutative Noetherian ring and $I$ be   an ideal of $R$. A prime ideal $\mathfrak{p}\subset  R$ is an {\it associated prime} of $I$ if there exists an element $v$ in $R$ such that $\mathfrak{p}=(I:_R v)$, where $(I:_R v)=\{r\in R |~  rv\in I\}$. The  {\it set of associated primes} of $I$, denoted by  $\mathrm{Ass}_R(R/I)$, is the set of all prime ideals associated to  $I$. 
The minimal members of $\mathrm{Ass}_R(R/I)$ are called the {\it minimal} primes of $I$, and $\mathrm{Min}(I)$ denotes the set of minimal prime ideals of $I$. Moreover, the associated primes of $I$ which are not minimal are called the {\it embedded} primes of $I$. If $I$ is a square-free monomial ideal, then $\mathrm{Ass}_R(R/I)=\mathrm{Min}(I)$, for example see \cite[Corollary 1.3.6]{HH1}.  

\begin{definition}\label{spersistence}
The ideal $I$ is said to have the {\it persistence property} if $\mathrm{Ass}(R/I^k)\subseteq \mathrm{Ass}(R/I^{k+1})$ 
for all positive integers $k$. Moreover, an ideal $I$ satisfies the {\it strong persistence property} if $(I^{k+1}: I)=I^k$ for all positive integers $k$,  for more details refer to \cite{ HQ, N2}. The strong persistence property implies the persistence property, however the converse is not true, as noted in \cite{HQ}. 
\end{definition}

Here, we should recall  the definition of symbolic powers of an ideal. 

\begin{definition} (\cite[Definition 4.3.22]{V1})
 Let $I$  be an ideal of a ring  $R$  and  $\mathfrak{p}_1, \ldots, \mathfrak{p}_r$  the minimal primes of $I$. Given an integer $n \geq 1$, the {\it $n$-th symbolic
power} of $I$  is defined to be the ideal 
$$I^{(n)} = \mathfrak{q}_1  \cap \cdots \cap  \mathfrak{q}_r,$$
where $\mathfrak{q}_i$  is the primary component of $I^n$  corresponding to $\mathfrak{p}_i$.  
\end{definition}

Furthermore,  we say that $I$  has the {\it symbolic strong persistence property}  if $(I^{(k+1)}: I^{(1)})=I^{(k)}$ for all $k$, where $I^{(k)}$ denotes the  $k$-th symbolic  power  of $I$, cf. \cite{KNT, RT}. \par 
 
 Let $R$ be a unitary commutative ring and $I$ an ideal in $R$. An element $f\in R$ is {\it integral} over $I$, if there exists an equation 
 $$f^k+c_1f^{k-1}+\cdots +c_{k-1}f+c_k=0 ~~\mathrm{with} ~~ c_i\in I^i.$$
 The set of elements $\overline{I}$ in $R$ which are integral over $I$ is the 
 {\it integral closure} of $I$. The ideal $I$ is {\it integrally closed}, if $I=\overline{I}$, and $I$ is {\it normal} if all powers of $I$ are integrally closed, we refer to \cite{HH1} for more information. \par 
 
 In particular, if $I$ is a monomial ideal, then
the notion of integrality becomes simpler, namely, a monomial 
$u \in  R=K[x_1, \ldots, x_n]$ is integral
over $I\subset R$ if and only if there exists an integer $k$ such that  $u^k \in I^k$, see \cite[Theorem 1.4.2]{HH1}.

An ideal $I$ is called {\it normally torsion-free} if  $\mathrm{Ass}(R/I^k)\subseteq \mathrm{Ass}(R/I)$, for all $k\geq 1$. If $I$ is a square-free monomial ideal, then $I$ is normally torsion-free if and only if $I^k=I^{(k)}$, for all $k \geq 1$, see \cite[Theorem 1.4.6]{HH1}. 

\begin{definition}

 A monomial ideal $I$ in a polynomial  ring $R=K[x_1, \ldots, x_n]$ over a field $K$ is called {\it nearly normally torsion-free}  if there exist a positive integer $k$ and a monomial prime ideal
 $\mathfrak{p}$ such that $\mathrm{Ass}_R(R/I^m)=\mathrm{Min}(I)$ for all $1\leq m\leq k$, and 
 $\mathrm{Ass}_R(R/I^m) \subseteq \mathrm{Min}(I) \cup \{\mathfrak{p}\}$ for all $m \geq k+1$, see \cite[Definition 2.1]{Claudia}. 
 
\end{definition}

Next, we recall some notions which are related to graph theory and hypergraphs.

\begin{definition}
Let $G=(V(G), E(G))$ be a finite simple graph on the vertex set $V(G)=\{1,\ldots,n\}$.  Then   the {\it edge ideal } associated to $G$ is the monomial ideal
$$I(G)=(x_ix_j \ : \ \{i,j\}\in E(G)) 
\subset R=K[x_1,\ldots, x_n]. $$
\end{definition}

A finite {\it hypergraph} $\mathcal{H}$ on a vertex set $[n]=\{1,2,\ldots,n\}$ is a collection of edges  $\{ E_1, \ldots, E_m\}$ with $E_i \subseteq [n]$, for all $i=1, \ldots,m$. The vertex set $[n]$ of $\mathcal{H}$ is denoted by $V_{\mathcal{H}}$, and the edge set of $\mathcal{H}$ is denoted by $E_{\mathcal{H}}$. Typically, a hypergraph is represented as a pair $(V_{\mathcal{H}}, E_{\mathcal{H}})$. A hypergraph $\mathcal{H}$ is called {\it simple}, if $E_i \subseteq  E_j$ implies  $i = j$. Moreover, if $|E_i|=d$, for all $i=1, \ldots, m$, then $\mathcal{H}$ is called {\em $d$-uniform } hypergraph. A 2-uniform hypergraph $\mathcal{H}$ is just a finite simple graph. If $\mathcal{W}$ is a subset of the vertices of $\mathcal{H}$, then the {\it induced subhypergraph} of $\mathcal{H}$  on $\mathcal{W}$ is $(\mathcal{W}, E_{\mathcal{W}})$, where $E_{\mathcal{W}}=\{E\cap \mathcal{W}: E \in E_{\mathcal{H}} \text{ and } E \cap \mathcal{W} \neq \emptyset\}$.

In \cite{HaV}, H\`a   and Van Tuyl extended the concept of the edge ideal  to hypergraphs. 
 
 \begin{definition} \cite{HaV}
Let $\mathcal{H}$ be  a hypergraph on the vertex set $V_{\mathcal{H}}=[n]$,  and $E_{\mathcal{H}} = \{E_1, \ldots, E_m\}$ be  the edge set of $\mathcal{H}$. Then the {\it edge ideal} corresponding
to $\mathcal{H}$ is given by
$$I(\mathcal{H}) = (\{x^{E_i}~ | ~E_i\in  E_{\mathcal{H}}\}),$$
where $x^{E_i}=\prod_{j\in E_i} x_j$.

A subset  $W \subseteq V_{\mathcal{H}}$  is a {\it vertex cover} of $\mathcal{H}$ if $W \cap E_i\neq \emptyset$  for all $i=1, \ldots, m$. A vertex cover $W$ is {\it minimal} if no proper subset of $W$ is a vertex cover of  $\mathcal{H}$.  Let $W_1, \ldots, W_t$ be  the minimal vertex covers of $\mathcal{H}$.  Then  the cover ideal of the hypergraph  $\mathcal{H}$, 
denoted by $J(\mathcal{H})$, is given by 
$J(\mathcal{H})=(X_{W_1}, \ldots, X_{W_t}),$ where 
$X_{W_j}=\prod_{r\in W_j}x_r$ for each $j=1, \ldots, t$.
  \end{definition}
Moreover, recall that the Alexander dual of a square-free monomial ideal $I$, denoted by $I^\vee$,  is given by 
 $$I^\vee= \bigcap_{u\in \mathcal{G}(I)} (x_i~:~ x_i|u).$$

In particular, according to  Proposition 2.7 in \cite{FHM}, one has  $J(\mathcal{H})=I(\mathcal{H})^{\vee}$, where  
$I(\mathcal{H})^{\vee}$ denotes the Alexander dual of $I(\mathcal{H})$.

Throughout this paper,  we denote the unique minimal set of monomial generators of a  monomial ideal $I$  by $\mathcal{G}(I)$. Also, $R=K[x_1,\ldots, x_n]$ is a polynomial ring over a field $K$, $\mathfrak{m}=(x_1, \ldots, x_n)$ is  the  graded maximal ideal of $R$, and $x_1, \ldots, x_n$ are indeterminates. 


\section{Some classes of nearly normally torsion-free ideals}\label{generalresults}

In this section, our aim is to give additional  classes of nearly normally torsion-free monomial ideals. To achieve this, we first recall the definition of the monomial localization of a monomial ideal with respect to a monomial prime ideal. 
Let $I$ be a monomial ideal in a polynomial ring $R=K[x_1, \ldots, x_n]$ over a field $K$.
  We denote by $V^*(I)$ the set of monomial prime ideals containing $I$. Let $\mathfrak{p}=(x_{i_1}, \ldots, x_{i_r})$ be a monomial prime ideal. The {\it monomial localization} of $I$ with respect to $\mathfrak{p}$, denoted by $I(\mathfrak{p})$, is the ideal in the polynomial ring $R(\mathfrak{p})=K[x_{i_1}, \ldots, x_{i_r}]$  which is obtained from $I$ by applying the $K$-algebra homomorphism $R\rightarrow R(\mathfrak{p})$  with $x_j\mapsto 1$  for all $x_j\notin \{x_{i_1}, \ldots, x_{i_r}\}$.

\begin{lemma} \label{Lem. 1}
Let $I$ be a nearly normally torsion-free monomial ideal in a polynomial ring $R=K[x_1, \ldots, x_n]$ over a field $K$. Then   there exists a monomial prime ideal $\mathfrak{p}\in V^*(I)$ such that $I(\mathfrak{p}\setminus \{x_i\})$ is normally torsion-free  for all $x_i\in \mathfrak{p}$.
\end{lemma}

\begin{proof}
Since  $I$ is  nearly normally torsion-free, there exists a positive integer $k$ and a monomial prime ideal  $\mathfrak{p}$ such that    $\mathrm{Ass}_R(R/I^m)=\mathrm{Min}(I)$ for all $1\leq m\leq k$, and 
 $\mathrm{Ass}_R(R/I^m) \subseteq \mathrm{Min}(I) \cup \{\mathfrak{p}\}$ for all $m \geq k+1$. We claim that $I(\mathfrak{p}\setminus \{x_i\})$ is normally torsion-free  for all $x_i\in \mathfrak{p}$. To do this, fix  $x_i\in \mathfrak{p}$, and  set $\mathfrak{q}:=\mathfrak{p}\setminus \{x_i\}$. 
 We need to show that 
 $\mathrm{Ass}_{R(\mathfrak{q})}(R(\mathfrak{q})/(I(\mathfrak{q}))^\ell) \subseteq \mathrm{Ass}_{R(\mathfrak{q})}(R(\mathfrak{q})/I(\mathfrak{q}))$ for all $\ell$. Fix $\ell \geq 1$. 
It follows from \cite[Lemma 4.6]{RNA} that
$$\mathrm{Ass}_{R(\mathfrak{q})}(R(\mathfrak{q})/(I(\mathfrak{q}))^\ell)=
\mathrm{Ass}_{R(\mathfrak{q})}(R(\mathfrak{q})/I^\ell(\mathfrak{q}))=\{Q~:~ Q\in \mathrm{Ass}_{R}(R/I^\ell)~\mathrm{and}~ Q \subseteq \mathfrak{q}\}.$$ 
Pick an arbitrary element $Q\in \mathrm{Ass}_{R(\mathfrak{q})}(R(\mathfrak{q})/(I(\mathfrak{q}))^\ell)$. Thus, $Q\in \mathrm{Ass}_{R}(R/I^\ell)$ and  $ Q \subseteq \mathfrak{q}$. Since $\mathfrak{q}=\mathfrak{p}\setminus \{x_i\}$, this yields that $Q\neq \mathfrak{p}$; thus, one must have  $Q\in \mathrm{Min}(I)$. Hence,  $Q\in   \mathrm{Ass}_R(R/I)$, and so   $Q\in  \mathrm{Ass}_{R(\mathfrak{q})}(R(\mathfrak{q})/I(\mathfrak{q}))$. Therefore,  
$I(\mathfrak{p}\setminus \{x_i\})$ is normally torsion-free. 
\end{proof}

 
\begin{lemma} \label{Lem. 2}
Let $I$ be a   monomial ideal in a polynomial ring $R=K[x_1, \ldots, x_n]$ over a field $K$ such that $\mathrm{Ass}_R(R/I)=\mathrm{Min}(I)$.  Let   $I(\mathfrak{m}\setminus \{x_i\})$ be  normally torsion-free  for all $i=1, \ldots, n$, where $\mathfrak{m}=(x_1, \ldots, x_n)$. Then $I$ is nearly normally torsion-free. 
\end{lemma}

\begin{proof}
In the light of   $\mathrm{Ass}_R(R/I)=\mathrm{Min}(I)$, it is enough for us to  show that  $\mathrm{Ass}_R(R/I^k) \subseteq \mathrm{Min}(I) \cup \{\mathfrak{m}\}$ for all $k \geq 2.$ 
To achieve this, fix $k \geq 2$, and take an arbitrary element $Q \in \mathrm{Ass}_R(R/I^k)$. If $Q=\mathfrak{m}$, then the proof is over. Hence, let $Q\neq \mathfrak{m}$. 
On account of $Q$ is a monomial prime ideal, this implies that   $Q\subseteq \mathfrak{m}\setminus \{x_j\}$ for some $x_j \in \mathfrak{m}$.
Put  $\mathfrak{q}:=\mathfrak{m}\setminus \{x_j\}$. Because $I(\mathfrak{q})$ is normally torsion-free, we thus have  $\mathrm{Ass}_{R(\mathfrak{q})}(R(\mathfrak{q})/(I(\mathfrak{q}))^k) \subseteq \mathrm{Ass}_{R(\mathfrak{q})}(R(\mathfrak{q})/I(\mathfrak{q}))$.  Also, one can deduce from \cite[Lemma 4.6]{RNA} that  $Q\in \mathrm{Ass}_{R(\mathfrak{q})}(R(\mathfrak{q})/(I(\mathfrak{q}))^k)$, and so
$Q \in \mathrm{Ass}_{R(\mathfrak{q})}(R(\mathfrak{q})/I(\mathfrak{q}))$. This yields that  $Q \in\mathrm{Ass}_R(R/I)$, and hence  $Q\in \mathrm{Min}(I)$.  
This completes the proof. 
\end{proof}

As an immediate consequence of Lemma \ref{Lem. 2},  we obtain the following corollary:

\begin{corollary} \label{Cor. 1}
Let $I$ be a square-free  monomial ideal in a polynomial ring $R=K[x_1, \ldots, x_n]$ over a field $K$.  Let   $I(\mathfrak{m}\setminus \{x_i\})$ be  normally torsion-free  for all $i=1, \ldots, n$, where $\mathfrak{m}=(x_1, \ldots, x_n)$. Then $I$ is nearly normally torsion-free. 
\end{corollary}

As an application of Lemma \ref{Lem. 2}, we give a class of nearly normally torsion-free monomial ideals in the subsequent  
theorem. To see this, one needs to recall from \cite{HV} that a monomial ideal is said to be a {\it monomial ideal of intersection type} when it can be presented as an intersection of powers of monomial prime ideals.

\begin{proposition} \label{App. 1}
 Let $R=K[x_1, \ldots, x_n]$ be a polynomial ring and 
 $I=\cap_{\mathfrak{p}\in \mathrm{Ass}_R(R/I)}\mathfrak{p}^{d_{\mathfrak{p}}}$  be a monomial ideal of intersection type such that for any $1\leq i \leq n$, there exists unique  
 $\mathfrak{p}\in \mathrm{Ass}_R(R/I)$ with  $I(\mathfrak{m}\setminus \{x_i\})=\mathfrak{p}^{d_{\mathfrak{p}}}$, where $\mathfrak{m}=(x_1, \ldots, x_n)$. Then $I$ is nearly normally torsion-free. 
 \end{proposition}

\begin{proof}
 We first note that the assumption concludes  that  the monomial ideal $I$ must have  the following form 
 $$I=(\mathfrak{m}\setminus \{x_1\})^{d_1} \cap (\mathfrak{m}\setminus \{x_2\})^{d_2} \cap \cdots \cap  (\mathfrak{m}\setminus \{x_n\})^{d_n},$$ 
 for some positive integers $d_1, \ldots, d_n$. 
 This gives rise to  $\mathrm{Ass}_R(R/I)=\mathrm{Min}(I)$.  In addition, since $(\mathfrak{m}\setminus \{x_i\})^{d_i}$, for each $i=1, \ldots, n$, is  normally torsion-free, the claim can be deduced promptly from Lemma \ref{Lem. 2}. 
\end{proof}


To show Lemma \ref{Lem. Multipe}, we require the following  auxiliary result. 

\begin{theorem} (\cite[Theorem 5.2]{KHN2})\label{5.2KHN}
Let $I$ be a monomial ideal of $R$ with $\mathcal{G}(I) =\{u_1,\ldots,u_m\}$. Also assume that there exists a monomial $h=x_{j_1}^{b_1}\cdots x_{j_s}^{b_s}$ such that $h | u_i$ for all $i=1,\ldots,m$. By setting $J:=(u_1/h,\ldots,u_m/h)$, we have   $$\mathrm{Ass}_R(R/I)=\mathrm{Ass}_R(R/J)\cup\{ (x_{j_1}),\ldots,(x_{j_s})\}.$$
\end{theorem}
The next   lemma says that  a monomial ideal is nearly normally torsion-free if and only if its  monomial multiple  is nearly normally torsion-free under certain conditions. It is an updated version of 
\cite[Lemma 3.5]{HLR}  and  \cite[Lemma 3.12]{SN}. 

\begin{lemma}\label{Lem. Multipe}
Let  $I$ be a monomial ideal in a polynomial ring $R=K[x_1, \ldots, x_n]$, and $h$  be a monomial in $R$ such that $\mathrm{gcd}(h,u)=1$ for all $u\in \mathcal{G}(I)$. Then  $I$ is nearly normally torsion-free  if and only if $hI$ is nearly normally torsion-free. 
\end{lemma}
\begin{proof}
$(\Rightarrow)$ Assume that $I$ is nearly normally torsion-free.  Let $h=x_{j_1}^{b_1}\cdots x_{j_s}^{b_s}$ with $j_1, \ldots, j_s \in \{1, \ldots, n\}$. On account of  Theorem \ref{5.2KHN}, we obtain, for all $\ell$,  
\begin{equation}
\mathrm{Ass}_R(R/(hI)^{\ell})=\mathrm{Ass}_R(R/I^{\ell})\cup\{ (x_{j_1}),\ldots,(x_{j_s})\}. \label{12}
\end{equation}
Due to  $\mathrm{gcd}(h,u)=1$ for all $u\in \mathcal{G}(I)$, it is routine to check that 
\begin{equation}
\mathrm{Min}(hI)=\mathrm{Min}(I)\cup\{ (x_{j_1}),\ldots,(x_{j_s})\}. \label{13}
\end{equation}
Since  $I$ is nearly normally torsion-free,  there exist a positive integer $k$ and a monomial prime ideal
 $\mathfrak{p}$ such that  
 $\mathrm{Ass}_R(R/I^m)=\mathrm{Min}(I)$ for all $1\leq m\leq k$, and 
 $\mathrm{Ass}_R(R/I^m) \subseteq \mathrm{Min}(I) \cup \{\mathfrak{p}\}$ for all $m \geq k+1$. Select an arbitrary element $\mathfrak{q}\in \mathrm{Ass}_R(R/(hI)^m)$. 
 Then \eqref{12} implies that $\mathfrak{q}\in \mathrm{Ass}_R(R/I^m)\cup\{ (x_{j_1}),\ldots,(x_{j_s})\}$. 
  If $1\leq m \leq k$, then \eqref{13} yields that 
 $\mathfrak{q}\in \mathrm{Min}(hI)$. Hence, let  $m\geq k+1$. The claim follows readily from \eqref{12}, \eqref{13}, and $\mathrm{Ass}_R(R/I^m) \subseteq \mathrm{Min}(I) \cup \{\mathfrak{p}\}$. 
 
$(\Leftarrow)$  Let $hI$ be nearly normally torsion-free. 
This means that there exist a positive integer $k$ and a monomial prime ideal
 $\mathfrak{p}$ such that  
 $\mathrm{Ass}_R(R/(hI)^m)=\mathrm{Min}(hI)$ for all $1\leq m\leq k$, and 
 $\mathrm{Ass}_R(R/(hI)^m) \subseteq \mathrm{Min}(hI) \cup \{\mathfrak{p}\}$ for all $m \geq k+1$. Take  an arbitrary element $\mathfrak{q}\in \mathrm{Ass}_R(R/I^m)$. 
Because $\mathfrak{q}\in \mathrm{Ass}_R(R/I^m)$, we get $\mathfrak{q} \notin 
 \{ (x_{j_1}),\ldots,(x_{j_s})\}$. 
If $1\leq m \leq k$, then \eqref{12} gives  that 
 $\mathfrak{q}\in \mathrm{Ass}_R(R/(hI)^m)$, and so $\mathfrak{q}\in \mathrm{Min}(hI)$. As $\mathfrak{q} \notin 
 \{ (x_{j_1}),\ldots,(x_{j_s})\}$, we gain  $\mathfrak{q}\in \mathrm{Min}(I)$. We thus assume that $m\geq k+1$. 
 One can derive the assertion according to the facts $\mathfrak{q} \notin  \{ (x_{j_1}),\ldots,(x_{j_s})\}$, 
 \eqref{12}, \eqref{13}, and $\mathrm{Ass}_R(R/(hI)^m) \subseteq \mathrm{Min}(hI) \cup \{\mathfrak{p}\}$. 
\end{proof}

We conclude this section  by observing that nearly normally torsion-freeness does not imply the strong persistence property. It is not known if nearly normally torsion-freeness implies the persistence property.

\begin{example}
Let $R=K[x_1, x_2, x_3]$ be the polynomial ring  over a field $K$ and $I:=(x_2^4, x_1x_2^3, x_1^3x_2, x_1^4x_3)$  be a monomial  ideal of  $R$.  On account of  
$$
I=(x_1, x_2^4) \cap (x_1^3, x_2^3) \cap (x_1^4, x_2) \cap (x_2, x_3), $$ 
one can conclude that  
$\mathrm{Ass}_R(R/I)=\mathrm{Min}(I)=\{(x_1, x_2),     (x_2, x_3)\}.$
We claim that $\mathrm{Ass}_R(R/I^m)=\mathrm{Min}(I) \cup
\{(x_1, x_2, x_3)\}$ for all $m\geq 2$. To prove  this claim, fix $m\geq 2$.  In what follows, we verify that 
$(x_1, x_2, x_3)\in \mathrm{Ass}_R(R/I^m)$. To establish this, we show  that 
$(I^m:_Rv)=(x_1, x_2, x_3)$, where $v:=x_1^{3m-1} x_2^{m+1}$. To see this, one may  consider the following statements:
\begin{itemize}
\item[(i)] Since   $vx_1=x_1^{3m}x_2^{m+1}=
(x_1^3x_2)^mx_2$ and $x_1^3x_2\in I$, we get $vx_1\in I^m,$ and so  $x_1 \in (I^m:_Rv)$;
\item[(ii)] Due to  $vx_2=x_1^{3m-1}x_2^{m+2}=
(x_1^3x_2)^{m-1} (x_1^2x_2^3)$ and $x_1^2x_2^3\in I$, one has $vx_2 \in I^m,$ and hence   $x_2 \in (I^m:_Rv)$;
\item[(iii)] As  $vx_3=x_1^{3m-1}x_2^{m+1}x_3 =(x_1^3x_2)^{m-2}(x_1^5x_2^3x_3)$ and $x_1^5x_2^3x_3\in I^2$, we obtain $vx_3 \in I^m,$ and thus  $x_3 \in (I^m:_Rv)$.
\end{itemize} 
    Consequently, $(x_1, x_2, x_3) \subseteq (I^m:_Rv).$   For the converse,  let $v \in I^m$. This leads to 
there exist monomials $h_1, \ldots, h_m \in \mathcal{G}(I)$ such that   $x_3\nmid h_i$ for each $i=1, \ldots, m$, and  $h_1 \cdots h_m |  x_1^{3m-1} x_2^{m+1}$.  Hence, we have the following equality 
\begin{equation}
h_1 \cdots h_m=(x_2^4)^{{\alpha}_1} (x_1x_2^3)^{{\alpha}_2} (x_1^3x_2)^{{\alpha}_3}, 
\label{14}
\end{equation}
for some nonnegative integers $ {\alpha}_1, \alpha_2,  {\alpha}_3$ with ${\alpha}_1 +  \alpha_2 +  {\alpha}_3=m.$
Especially, one can derive from (\ref{14}) that 
\begin{equation}
4\alpha_1 + 3\alpha_2 + \alpha_3  \leq m+1, \label{15}
\end{equation}
and 
\begin{equation}
\alpha_2 + 3\alpha_3  \leq 3m-1. \label{16}
\end{equation}
Since $\sum_{i=1}^3 \alpha_i =m$, we obtain from (\ref{15}) that 
$m+3\alpha_1 + 2 \alpha_2  \leq m+1$, and so $3\alpha_1 + 2\alpha_2 \leq 1$. We thus have 
$\alpha_1= \alpha_2=0$, and so $\alpha_3 =m$. Moreover, it follows from (\ref{16}) that $3m  \leq 3m-1$,   which is a contradiction. Therefore, $v \notin I^m$, and so $(I^m:_Rv) =
(x_1, x_2, x_3)$. This gives rise to  $(x_1, x_2, x_3) \in 
\mathrm{Ass}_R(R/I^m)$ for all $m\geq 2$. Hence, 
$\mathrm{Ass}_R(R/I^m)=\mathrm{Min}(I) \cup
\{(x_1, x_2, x_3)\}$ for all $m\geq 2$. This 
means that $I$ has the persistence property and is a nearly normally torsion-free ideal.
On the other hand, using Macaulay2 \cite{GS} shows that 
$(I^2:_RI) \neq I$. We therefore deduce that $I$ does not satisfy the strong persistence property. 
\end{example}


\section{Nearly normally torsion-freeness and cover ideals}\label{coverideals}

In this section, our goal is to characterize all finite simple connected graphs such that their cover ideals are nearly normally torsion-free. To do this, one has to recall some results which will be used in the proof of  Theorem \ref{Main. 1}. We begin with the following lemma.

 \begin{lemma} \cite[Lemma 2.11]{FHV2} \label{FHV2}
  Let  $\mathcal{H}$ be a finite simple hypergraph on 
  $V = \{x_1, \ldots , x_n\}$ with cover ideal 
  $J(\mathcal{H}) \subseteq R=k[x_1, \ldots, x_n]$. 
  Then 
  $$P = (x_{i_1} , \ldots , x_{i_r}) \in \mathrm{Ass}(R/J(\mathcal{H})^d) \Leftrightarrow P = (x_{i_1} , \ldots , x_{i_r}) \in \mathrm{Ass}(k[P]/J(\mathcal{H}_P)^d),$$
where $k[P] =k[x_{i_1} , \ldots , x_{i_r}]$, and $\mathcal{H}_P$ is the induced hypergraph of $\mathcal{H}$ on the vertex set $P = \{x_{i_1} , \ldots , x_{i_r}\} \subseteq V$.  
\end{lemma}

In the following proposition, we investigate  the associated primes of powers of the cover ideals of odd cycle graphs.

 \begin{proposition}\label{Pro. 1} \cite[Proposition 3.6]{NKA}
 Suppose that $C_{2n+1}$ is  a cycle graph on the  vertex set $[2n+1]$, $R=K[x_1, \ldots, x_{2n+1}]$ is a  polynomial ring over a field $K$,
 and $\mathfrak{m}$ is the unique homogeneous maximal ideal  of $R$. Then  $$\mathrm{Ass}_R(R/(J(C_{2n+1}))^s)= \mathrm{Ass}_R(R/J(C_{2n+1}))\cup \{\mathfrak{m}\},$$  for all $s\geq 2$. In particular, 
 $$\mathrm{Ass}^\infty(J(C_{2n+1}))=\{(x_i, x_{i+1})~: ~ i=1, \ldots 2n\}\cup\{(x_{2n+1}, x_1)\}\cup \{\mathfrak{m}\}.$$
  \end{proposition}


  Next   theorem   explores  the relation between associated primes of powers of the cover ideal of the union of two   finite simple  graphs  with the associated primes of powers of the cover ideals of each of them, under the condition that they have only one common vertex.

 \begin{theorem} \cite[Theorem 11]{KNT} \label{IntersectionOne}
Let $G=(V(G), E(G))$ and $H=(V(H), E(H))$ be two finite simple connected graphs such that $|V(G)\cap V(H)|=1$. Let $L=(V(L), E(L))$  be the  finite simple graph such that   $V(L):=V(G) \cup V(H)$ and  $E(L):=E(G) \cup E(H)$.
 Then 
$$\mathrm{Ass}_{R}(R/J(L)^s)=\mathrm{Ass}_{R_1}(R_1/J(G)^s)\cup \mathrm{Ass}_{R_2}(R_2/J(H)^s),$$ for all  $s$, where $R_1=K[ x_\alpha : \alpha\in V(G)]$, 
$R_2=K[ x_\alpha : \alpha\in V(H)]$,
 and  $R=K[ x_\alpha : \alpha\in V(L)]$. 
\end{theorem} 


The subsequent  theorem  examines  the relation between  associated primes of powers of the cover ideal of the union of two   finite simple connected graphs   with the  associated primes of powers of the cover ideals of each of them, under the condition  that they have only one edge  in common.

\begin{theorem} \cite[Theorem 12]{KNT} \label{IntersectionTwo}
Let $G=(V(G), E(G))$ and $H=(V(H), E(H))$ be two finite simple connected graphs such that $|V(G)\cap V(H)|=2$ and 
$|E(G)\cap E(H)|=1$.  Let $L=(V(L), E(L))$  be the  finite simple graph such that   $V(L):=V(G) \cup V(H)$ and  $E(L):=E(G) \cup E(H)$. 
 Then 
$$\mathrm{Ass}_{R}(R/J(L)^s)=\mathrm{Ass}_{R_1}(R_1/J(G)^s)\cup \mathrm{Ass}_{R_2}(R_2/J(H)^s),$$ for all  $s$, where $R_1=K[ x_\alpha : \alpha\in V(G)]$, 
$R_2=K[ x_\alpha : \alpha\in V(H)]$,
 and  $R=K[ x_\alpha : \alpha\in V(L)]$. 
\end{theorem} 


To understand the proof of Theorem \ref{Main. 1}, we first review some notation  from \cite{JS} as follows:

Let $G=(V(G), E(G))$ be a finite simple connected graph. For any $x,y \in V(G)$, an $(x,y)$-path is simply a path between the vertices $x$ and $y$ in $G$. Also, for a vertex  subset $W$ of a graph $G$, $\langle W \rangle$ will denote the subgraph of $G$ induced by $W$. 

Let $G$ be an {\it almost bipartite} graph, that is to say, $G$ has only one induced odd cycle subgraph, say $C_{2k+1}$.   For each $i\in V(C_{2k+1})$, let 
$$A_i=\{x\in V(G)|~x\neq i ~\text{and for all~} j\in V(C_{2k+1}), ~i \text{~appears on every~}  (x,j)\text{-path}\}.$$
Based on \cite[Page 540]{JS}, it should be noted that the set $A_i$ may be empty for some $i$, and it follows from \cite[Lemma 2.3]{JS}  that if $A_i\neq \emptyset$, then the induced subgraph $\langle A_i \rangle$ is bipartite in its own right. Also, for every edge $e=\{i,j\} \in E(C_{2k+1})$, let 
\begin{align*}
B_e=  \{& x\in V(G)\setminus V(C_{2k+1})|~\text{for every}~ 
m\in V(C_{2k+1})\setminus \{i,j\}, ~\text{there } \\
& \text{is an}~ (x,m)\text{-path in which} ~ i \text{~appears but}~ j~\text{does not, and an } \\
&   (x,m)\text{-path in which} ~ j \text{~appears but}~ i~ \text{does not}\}.
\end{align*} 
Once again, according to \cite[Page 541]{JS}, one must note that the set $B_e$ may be empty for some $e$, and it can be deduced  from \cite[Lemma 2.3]{JS}  that if $B_e\neq \emptyset$, then the induced subgraph $\langle B_e \rangle$ is bipartite in its own right. Furthermore, the $A_i$'s and
$B_e$'s are all mutually disjoint. In other words, each vertex of $G$ is in exactly one of the following sets: (i) $V(C_{2k+1})$, 
(ii) $A_i$ for some $i$, and (iii) $B_e$ for some $e$. Moreover, for each $i,j\in V(C_{2k+1})$ with $i\neq j$, and $e, e'\in E(C_{2k+1})$ with $e\neq e'$, one can easily derive from the definitions  that there is no path between any two vertices of $\langle A_i \rangle$ and $\langle B_e \rangle$, or $\langle A_i \rangle$ and $\langle A_j \rangle$, or  $\langle B_e \rangle$ and $\langle B_{e'} \rangle$.  

As an example, consider the following graph $G$ from \cite{JS}.

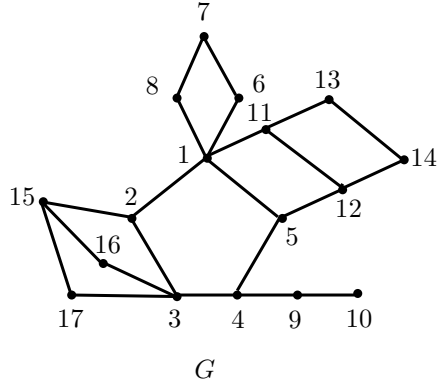
  \begin{center}
 
\scalebox{1} 
{
\begin{pspicture}(0,-2.5629687)(6.0628123,2.5629687)
\psdots[dotsize=0.12](2.7809374,0.42453125)
\psdots[dotsize=0.12](3.7809374,-0.37546876)
\psdots[dotsize=0.12](1.7809376,-0.37546876)
\psdots[dotsize=0.12](3.1809375,-1.3954687)
\psdots[dotsize=0.12](2.3809376,-1.4154687)
\psline[linewidth=0.04cm](2.7609375,0.42453125)(3.7209375,-0.35546875)
\psline[linewidth=0.04cm](2.7609375,0.42453125)(1.7809376,-0.35546875)
\psline[linewidth=0.04cm](3.7409375,-0.37546876)(3.1809375,-1.3354688)
\psline[linewidth=0.04cm](1.7809376,-0.39546874)(2.3609376,-1.3754687)
\psline[linewidth=0.04cm](2.3809376,-1.3954687)(3.1409376,-1.3954687)
\psdots[dotsize=0.12](3.9809375,-1.3954687)
\psline[linewidth=0.04cm](3.1809375,-1.4154687)(3.2009375,-1.3954687)
\psline[linewidth=0.04cm](3.1809375,-1.3954687)(3.9809375,-1.3954687)
\psline[linewidth=0.04cm](3.9609375,-1.4154687)(3.9609375,-1.3954687)
\psdots[dotsize=0.12](4.7809377,-1.3754687)
\psline[linewidth=0.04cm](3.9809375,-1.3954687)(4.7409377,-1.3954687)
\psdots[dotsize=0.12](4.5809374,0.00453125)
\psdots[dotsize=0.12](5.4009376,0.40453124)
\psdots[dotsize=0.12](3.5609374,0.8045313)
\psdots[dotsize=0.12](4.4009376,1.2045312)
\psdots[dotsize=0.12](3.2009375,1.2245313)
\psdots[dotsize=0.12](2.7409375,2.0445313)
\psdots[dotsize=0.12](2.3809376,1.2245313)
\psdots[dotsize=0.12](1.4009376,-0.97546875)
\psdots[dotsize=0.12](0.6009375,-0.15546875)
\psline[linewidth=0.04cm](0.5809375,-0.15546875)(1.7409375,-0.35546875)
\psline[linewidth=0.04cm](0.6009375,-0.17546874)(1.3609375,-0.9554688)
\psline[linewidth=0.04cm](1.3809375,-0.9554688)(2.3609376,-1.4154687)
\psline[linewidth=0.04cm](2.7409375,2.0645313)(2.3809376,1.2445313)
\psline[linewidth=0.04cm](2.7409375,2.0245314)(3.1609375,1.2645313)
\psline[linewidth=0.04cm](2.3809376,1.2245313)(2.7609375,0.44453126)
\psline[linewidth=0.04cm](3.2009375,1.2245313)(2.7809374,0.44453126)
\psline[linewidth=0.04cm](2.7609375,0.44453126)(3.5409374,0.8045313)
\psline[linewidth=0.04cm](3.5609374,0.82453126)(4.4009376,1.2045312)
\psline[linewidth=0.04cm](4.3809376,1.2245313)(5.3809376,0.40453124)
\psline[linewidth=0.04cm](3.7809374,-0.35546875)(4.5609374,0.00453125)
\psline[linewidth=0.04cm](4.5609374,0.04453125)(4.5809374,0.08453125)
\psline[linewidth=0.04cm](4.5809374,0.02453125)(5.4009376,0.40453124)
\psline[linewidth=0.04cm](3.5609374,0.8045313)(4.5609374,0.04453125)
\usefont{T1}{ptm}{m}{n}
\rput(2.4723437,0.47453126){$1$}
\usefont{T1}{ptm}{m}{n}
\rput(1.7723438,-0.08546875){$2$}
\usefont{T1}{ptm}{m}{n}
\rput(2.3523438,-1.7054688){$3$}
\usefont{T1}{ptm}{m}{n}
\rput(3.1923437,-1.7254688){$4$}
\usefont{T1}{ptm}{m}{n}
\rput(3.9123437,-0.62546873){$5$}
\usefont{T1}{ptm}{m}{n}
\rput(3.4723437,1.4145312){$6$}
\usefont{T1}{ptm}{m}{n}
\rput(2.7323437,2.3745313){$7$}
\usefont{T1}{ptm}{m}{n}
\rput(2.0523438,1.3945312){$8$}
\usefont{T1}{ptm}{m}{n}
\rput(3.9523437,-1.7254688){$9$}
\usefont{T1}{ptm}{m}{n}
\rput(4.802344,-1.7054688){$10$}
\usefont{T1}{ptm}{m}{n}
\rput(3.4823437,1.0345312){$11$}
\usefont{T1}{ptm}{m}{n}
\rput(4.662344,-0.24546875){$12$}
\usefont{T1}{ptm}{m}{n}
\rput(4.382344,1.4745313){$13$}
\usefont{T1}{ptm}{m}{n}
\rput(5.662344,0.43453124){$14$}
\usefont{T1}{ptm}{m}{n}
\rput(0.32234374,-0.08546875){$15$}
\usefont{T1}{ptm}{m}{n}
\rput(1.4623437,-0.72546875){$16$}
\usefont{T1}{ptm}{m}{n}
\rput(2.7523437,-2.3854687){$G$}
\psdots[dotsize=0.12](0.9809375,-1.3954687)
\psline[linewidth=0.04cm](0.5809375,-0.17546874)(0.9609375,-1.3554688)
\psline[linewidth=0.04cm](0.9809375,-1.3954687)(2.3609376,-1.4154687)
\usefont{T1}{ptm}{m}{n}
\rput(0.96234375,-1.7054688){$17$}
\end{pspicture} 
}

 \end{center}
  
Direct computations show that  $A_1=\{6, 7, 8\}, ~A_4=\{9, 10\}, ~A_2=A_3=A_5=\emptyset,$
 and 
 $B_{\{1,5\}}=\{11, 12, 13, 14\},  ~
 B_{\{2,3\}}=\{15, 16, 17\},~ B_{\{1,2\}}=B_{\{3,4\}}=B_{\{4,5\}}=\emptyset.$
  
The following theorem is the first main result of this section.  

\begin{theorem} \label{Main. 1}
Assume that $G=(V(G), E(G))$ is a finite simple connected graph, and $J(G)$ denotes the cover ideal of $G$. Then $J(G)$ is nearly normally torsion-free if and only if $G$ is either a bipartite graph or an almost bipartite graph.
\end{theorem}

\begin{proof}
To show the forward implication, let $J(G)$ be nearly normally torsion-free. Suppose, on the contrary, that $G$ is neither bipartite nor almost bipartite. This gives that $G$ has at least two induced odd cycle subgraphs, say $C$ and $C'$. It follows from Proposition  \ref{Pro. 1} that $\mathfrak{p}=(x_j~:~ j \in  V(C))\in \mathrm{Ass}(J(C)^s)$ and  $\mathfrak{p}'=(x_j~:~ j \in  V(C'))\in \mathrm{Ass}(J(C')^s)$ for all $s\geq 2$. Since $G_{\mathfrak{p}}=C$ and $G_{\mathfrak{p}'}=C'$, 
we can deduce from Lemma \ref{FHV2} that $\mathfrak{p}, \mathfrak{p}' \in \mathrm{Ass}_R(R/J(G)^s)$ for all $s\geq 2$. This contradicts the assumption that $J(G)$ is nearly normally torsion-free. 

Conversely, if $G$ is bipartite, then on account of \cite[Corollary 2.6]{GRV}, one has $J(G)$ is normally torsion-free, and so $J(G)$ is nearly normally torsion-free. Next, we assume that $G$ is an almost bipartite graph, and let $C$ be its unique induced odd cycle subgraph. Put  $\mathfrak{p}=(x_j~:~ j\in V(C))$. We claim that $\mathrm{Ass}(J(G)^s)=\mathrm{Min}(J(G)) \cup \{\mathfrak{p}\}$ for all $s\geq 2$. Fix $s\geq 2$. For any $i\in V(C)$ and $e\in E(C)$, assume that $A_i$ and $B_e$ are the vertex subsets of $G$ as defined in the discussion above. Without loss of generality, suppose that $A_i \neq \emptyset$ for all $i=1, \ldots, r$ and $B_{e_j}\neq \emptyset$ for all $j=1, \ldots, t$. Set $H_i:=\langle A_i \cup \{i\}\rangle$ for all  
  $i=1, \ldots, r$, and $L_j:=\langle B_{e_j} \cup \{\alpha_j, \beta_j\}\rangle$ for all $j=1, \ldots, t$, where  $e_j=\{\alpha_j, \beta_j\}$. Accordingly, we get $|V(C) \cap V(H_i)|=1$ for all     $i=1, \ldots, r$, $|V(C) \cap V(L_j)|=2$ 
and $|E(C) \cap E(L_j)|=1$ for all $j=1, \ldots, t$. On the other hand, it should be noted that all $H_i$ and $L_j$ are bipartite, 
and so \cite[Corollary 2.6]{GRV} yields that  $\mathrm{Ass}(J(H_i)^s)=\mathrm{Min}(J(H_i))$ and   
$\mathrm{Ass}(J(L_j)^s)=\mathrm{Min}(J(L_j))$ for all 
$i=1, \ldots, r$ and $j=1, \ldots, t$. Now, repeated applications of Theorems \ref{IntersectionOne} and \ref{IntersectionTwo} give  that $\mathrm{Ass}(J(G)^s)=\mathrm{Ass}(J(C)^s) \cup \mathrm{Min}(J(H_i)) \cup \mathrm{Min}(J(L_j))$ for all 
$i=1, \ldots, r$ and $j=1, \ldots, t$. By virtue of Proposition \ref{Pro. 1}, one obtains $\mathrm{Ass}(J(C)^s)=\{\mathfrak{p}\} \cup \mathrm{Min}(J(C))$. 
We thus have $\mathrm{Ass}(J(G)^s)= \mathrm{Min}(J(G)) \cup \{\mathfrak{p}\}$, as claimed. This shows that $J(G)$ is nearly normally torsion-free, and the proof is done. 
\end{proof}


Now, we focus on cover ideals of hypergraphs. For this purpose, we state the following theorem. To do this,  we recall  some definitions   which  will be 
necessary for  understanding Theorem \ref{NTF1}.  We list them as follows:

\begin{definition} (see \cite[Definition 2.7]{FHV2})
 Let  $\mathcal{H}= (V_{\mathcal{H}} , E_{\mathcal{H}})$ be a hypergraph. A {\em $d$-coloring} of $\mathcal{H}$ is any partition of $V_{\mathcal{H}} = C_1\cup \cdots \cup C_d$
into $d$ disjoint sets such that for every $E \in E_{\mathcal{H}}$, we have $E\nsubseteq C_i$  for all $i = 1, \ldots ,d$. (In the case of a
graph $G$, this simply means that any two vertices connected by an edge receive different colors.) The
$C_i$'s are called the color classes of $\mathcal{H}$. Each color class $C_i$ is an {\em independent set}, meaning that $C_i$ does not
contain any edge of the hypergraph. The chromatic number of $\mathcal{H}$, denoted by $\chi(\mathcal{H})$, is the minimal $d$
such that $\mathcal{H}$  has a $d$-coloring.
\end{definition}
\begin{definition} (see \cite[Definition 2.8]{FHV2})
 A hypergraph $\mathcal{H}$ is called {\em critically $d$-chromatic} if 
$\chi(\mathcal{H})= d$, but for every vertex
$x\in V_{\mathcal{H}}$, $\chi(\mathcal{H}\setminus \{x\})< d$,   where 
$\mathcal{H}\setminus \{x\}$ denotes the hypergraph $\mathcal{H}$ with $x$ and all edges containing $x$ removed.
\end{definition}

\begin{definition} (see \cite[Definition 4.2]{FHV2})
 Let  $\mathcal{H}= (V_{\mathcal{H}} , E_{\mathcal{H}})$ be a hypergraph with $V_{\mathcal{H}}=\{x_1, \ldots, x_n\}$. For each $s$, the {\em $s$-th expansion} of $\mathcal{H}$ is defined to be the hypergraph obtained by replacing each vertex $x_i \in V_{\mathcal{H}}$ by a collection $\{x_{ij}~|~ j=1, \ldots, s\}$, and replacing $E_{\mathcal{H}}$ by the edge  set that  consists of edges 
 $\{x_{i_1l_1}, \ldots, x_{i_rl_r}\}$ whenever 
 $\{x_{i_1}, \ldots, x_{i_r}\}\in E_{\mathcal{H}}$ and edges 
 $\{x_{il}, x_{ik}\}$ for $l\neq k$. We denote this hypergraph by $\mathcal{H}^s$. The new variables $x_{ij}$ are called the shadows of $x_i$. The process of setting $x_{il}$ to equal to $x_i$ for all $i$ and $l$ is called the {\em depolarization}. 
 \end{definition}

\begin{theorem} \label{NTF1}
Assume that  $\mathcal{G}=(V(\mathcal{G}), E(\mathcal{G}))$ and  $\mathcal{H}=(V(\mathcal{H}), E(\mathcal{H}))$ are 
 two finite simple hypergraphs such that $V(\mathcal{H})=V(\mathcal{G})\cup \{w\}$ with $w\notin V(\mathcal{G})$,  and $E(\mathcal{H})=E(\mathcal{G}) \cup \{\{v,w\}\}$ for some vertex $v\in V(\mathcal{G})$. Then 
$$\mathrm{Ass}_{R'}(R'/J(\mathcal{H})^s)=\mathrm{Ass}_{R}(R/J(\mathcal{G})^s)\cup
\{(x_v, x_w)\},$$
 for all  $s$, where $R=K[ x_\alpha : \alpha\in V(\mathcal{G})]$ and 
$R'=K[ x_\alpha : \alpha\in V(\mathcal{H})]$.
\end{theorem}
\begin{proof}
For convenience of notation, set $I:=J(\mathcal{G})$ and $J:=J(\mathcal{H})$. We first prove  that 
$\mathrm{Ass}_{R}(R/I^s)\cup
\{(x_v, x_w)\}\subseteq \mathrm{Ass}_{R'}(R'/J^s)$ for all  $s$. Fix $s\geq 1$, and assume that  $\mathfrak{p}=(x_{i_1}, \ldots, x_{i_r})$ is an  arbitrary element of  $\mathrm{Ass}_R(R/I^s)$. According to 
\cite[Lemma 2.11]{FHV2}, we get  
$\mathfrak{p}\in \mathrm{Ass}(K[\mathfrak{p}]/J(\mathcal{G}_\mathfrak{p})^s)$, where $K[\mathfrak{p}]=K[x_{i_1}, \ldots, x_{i_r}]$ and  $\mathcal{G}_\mathfrak{p}$ is the induced subhypergraph  of $\mathcal{G}$ on the vertex set $\{i_1, \ldots, i_r\}\subseteq V(\mathcal{G})$. Since   $\mathcal{G}_\mathfrak{p}= \mathcal{H}_\mathfrak{p}$, we have  
$\mathfrak{p}\in \mathrm{Ass}(K[\mathfrak{p}]/J(H_\mathfrak{p})^s)$. This yields  that  $\mathfrak{p}\in \mathrm{Ass}_{R'}(R'/J^s)$. On account  of  $(x_v, x_w)\in \mathrm{Ass}_{R'}(R'/J^s)$, one derives $\mathrm{Ass}_{R}(R/I^s)\cup \{(x_v, x_w)\}\subseteq \mathrm{Ass}_{R'}(R'/J^s)$.  
To complete the proof, it is enough for us  to show  the reverse inclusion. Assume that $\mathfrak{p}=(x_{i_1}, \ldots, x_{i_r})$ is
 an  arbitrary element of  $\mathrm{Ass}_{R'}(R'/J^s)$ with 
 $\{i_1, \ldots, i_r\}\subseteq V(\mathcal{H})$. If $\{i_1, \ldots, i_r\}\subseteq V(\mathcal{G})$, then \cite[Lemma 2.11]{FHV2} implies   that $\mathfrak{p}\in \mathrm{Ass}_R(R/I^s)$, and the proof is done. Thus, let $w\in \{i_1, \ldots, i_r\}$. 
 It follows from  \cite[Corollary 4.5]{FHV2} that  the associated  primes of $J(\mathcal{H})^s$ will correspond to critical chromatic subhypergraphs of size $s+1$ in the $s$-th expansion of $\mathcal{H}$. This means that one can take the induced subhypergraph on the vertex set $\{i_1, \ldots, i_r\}$, and 
 then form the $s$-th expansion on this induced subhypergraph, and within this new hypergraph find a critical $(s+1)$-chromatic hypergraph.   Notice that since this expansion cannot have any critical chromatic subgraphs, this implies that $\mathcal{H}_{\mathfrak{p}}$ must be connected. Hence, 
 $v\in \{i_1, \ldots, i_r\}$.  Without loss of generality, one may  assume  that $i_1=v$ and $i_2=w$. 
 Thanks to  $w$ is only connected to $v$ in the hypergraph $\mathcal{H}$, and because  this induced subhypergraph is critical, if we remove the vertex $w$, one  can color the resulting hypergraph with at least $s$ colors. This leads to  that $w$ has to be adjacent to at least $s$ vertices. But the only thing $w$ is adjacent to is the shadows of $w$ and the shadows of $v$, and so one has a clique among these vertices.  Accordingly, 
 $w$ and its neighbors will form a clique of size $s+1$. Since 
   a clique is a critical graph, it follows that we do not   need any element of $\{i_3, \ldots, i_r\}$ or their shadows when making the critical $(s+1)$-chromatic hypergraph. Consequently, we obtain  $\mathfrak{p}=(x_v, x_w)$, as required.    
\end{proof}
 Before stating the next result, it should be noted that one can always view a square-free monomial ideal as the cover ideal of a 
 simple hypergraph. In fact, assume that  $I$ is  a square-free monomial ideal, and $I^\vee$ denotes its Alexander dual.
 Also, let $\mathcal{H}$ denote the hypergraph corresponding to $I^\vee$. Then, we have $I=J(\mathcal{H})$, where 
 $J(\mathcal{H})$ denotes the cover ideal of the hypergraph $\mathcal{H}$. Consult  \cite{FHM} for further details and information.

 For instance, consider the following square-free monomial ideal in the polynomial ring $R=K[x_1,x_2,x_3,x_4,x_5]$ over a field $K$, 
 $$I=(x_1x_2x_3, x_2x_3x_4, x_3x_4x_5, x_4x_5x_1, x_5x_1x_2).$$
 Then the Alexander dual of $I$ is given by 
 \begin{align*}
 I^\vee = &(x_1, x_2, x_3) \cap (x_2, x_3, x_4) \cap (x_3, x_4, x_5) \cap (x_4, x_5, x_1) \cap (x_5, x_1, x_2)\\
 = &(x_3x_5, x_2x_5, x_2x_4, x_1x_4, x_1x_3).
 \end{align*}
 Now, define  the  hypergraph $\mathcal{H}=(\mathcal{X, \mathcal{E}})$ with $\mathcal{X}=\{x_1,x_2,x_3,x_4,x_5\}$ and
 $$\mathcal{E}=\{\{x_3,x_5\}, \{x_2,x_5\}, \{x_2, x_4\}, \{x_1, x_4\}, \{x_1, x_3\}\}.$$ 
Then the edge ideal and cover ideal of the  hypergraph $\mathcal{H}$ are  given by 
$$I(\mathcal{H})=(x_3x_5, x_2x_5, x_2x_4, x_1x_4, x_1x_3),$$ 
and 
$$J(\mathcal{H})= I(\mathcal{H})^\vee
=(x_3,x_5) \cap (x_2,x_5) \cap (x_2, x_4) \cap (x_1, x_4) \cap (x_1, x_3).$$
 
 It is easy to see that $J(\mathcal{H})=I$, as claimed.
 

 We are in a position to provide the second main result of this section in the following corollary.

 \begin{corollary} \label{Cor. NFT1}
 Let $I$ be a normally torsion-free square-free  monomial ideal in $%
R=K[x_{1},\ldots ,x_{n}]$ with $\mathcal{G}(I) \subset R$. Then the ideal $L:=IS\cap (x_{n},x_{n+1})\subset  S=R[x_{n+1}]$ satisfies  the following statements:
\begin{itemize}
\item[(i)] $L$  is normally torsion-free. 
\item[(ii)]  $L$ is nearly normally torsion-free. 
\item[(iii)]  $L$ is normal.
\item[(iv)]  $L$ has the strong persistence proeprty. 
\item[(v)]  $L$ has the persistence property. 
\item[(vi)]  $L$ has the symbolic strong persistence property. 

\end{itemize}
 \end{corollary}
 \begin{proof}
(i)  In the light of the argument which has been stated above, we can   assume that   $I=J(\mathcal{H})$ such that the hypergraph 
 $\mathcal{H}$ corresponds to $I^\vee$,  where 
  $I^\vee$ denotes the Alexander dual of $I$. Fix $t\geq 1$. 
  It follows now from Theorem \ref{NTF1} that 
  $$\mathrm{Ass}_{S}(S/L^t)=\mathrm{Ass}_{R}(R/J(\mathcal{H})^t)\cup \{(x_n, x_{n+1})\}.$$
Since $I$ is normally torsion-free, one can deduce that 
$\mathrm{Ass}_{R}(R/J(\mathcal{H})^t)=\mathrm{Min}(J(\mathcal{H}))$, and so $\mathrm{Ass}_{S}(S/L^t)=\mathrm{Min}(J(\mathcal{H}))\cup
\{(x_n, x_{n+1})\}.$ Therefore, $\mathrm{Ass}_{S}(S/L^t)=\mathrm{Min}(L)$. This means that 
$L$ is normally torsion-free, as desired. \par 
(ii) It is obvious from  the definition of nearly normally torsion-freeness and (i). \par 
(iii) By virtue of    \cite[Theorem 1.4.6]{HH1}, every normally torsion-free square-free monomial ideal is normal. Now, the  assertion can be deduced from (i).  \par 
(iv) Due to   \cite[Theorem 6.2]{RNA}, every normal monomial ideal has the strong persistence property, and hence  the claim  follows readily from (iii). \par 
(v)  It is shown in \cite{HQ} that the strong persistence property implies the persistence property. Therefore, one  can   derive  the assertion from (iv). \par 
(vi)  Thanks to    \cite[Theorem 11]{RT}, the strong persistence property implies the symbolic strong persistence property, and  thus the claim is  an immediate consequence of (iv).   
 \end{proof}
 

\section{The case of $t$-spread principal Borel ideals}\label{borel}

In this section, we focus on normally torsion-free and nearly normally torsion-free $t$-spread principal Borel ideals. Let $R=K[x_1, \ldots, x_n]$ be a polynomial ring over a field $K$. Let $t$ be a positive integer. A monomial $x_{i_1} x_{i_2} \cdots x_{i_d} \in R$ with $i_1 \leq i_2 \leq \cdots  \leq i_d$ is called {\it $t$-spread} if $i_j -i_{j-1} \geq t$ for all $j=2, \ldots, d$.  A monomial ideal in $R$ is called a {\it $t$-spread monomial ideal} if it is generated by $t$-spread monomials. A 0-spread monomial ideal is just an ordinary monomial ideal, while a 1-spread monomial ideal is just a square-free monomial ideal. In the following text, we will assume that $t \geq 1$.

Let $I\subset R$ be a $t$-spread monomial ideal. Then $I$ is called a {\it $t$-spread strongly stable ideal} if for all $t$-spread monomials $u\in \mathcal{G}(I)$, all $j\in \mathrm{supp}(u)$
 and all $1\leq i <j$ such that $x_i(u/x_j)$  is a $t$-spread monomial, it follows that $x_i(u/x_j)\in I$. 
 
 \begin{definition}\label{boreldef}
 A monomial ideal $I\subset R$ is called a {\it $t$-spread principal Borel} if there exists a $t$-spread monomial $u\in \mathcal{G}(I)$ such that  $I$  is the smallest $t$-spread strongly stable ideal which contains $u$. We denote it as $I=B_t(u)$. It should be noted that for a $t$-spread monomial $u=x_{i_1} x_{i_2} \cdots x_{i_d} \in R$, we have $x_{j_1} x_{j_2} \cdots x_{j_d}\in \mathcal{G}(B_t(u))$ if and only if 
 $j_1\leq i_1, \ldots, j_d \leq i_d$ and $j_k - j_{k-1} \geq t$ for $k\in \{2, \ldots, d\}$. We refer the reader to \cite{EHQ} for more information. 
 \end{definition}
 
To see an application of Corollary \ref{Cor. 1}, we re-prove Proposition 4.4 from \cite{Claudia}. 

\begin{proposition} \label{App. 2}
Let $u=x_ix_n$ be a $t$-spread monomial in $R=K[x_1, \ldots, x_n]$ with $i\geq t$. Then $I=B_t(u)$ is  nearly normally  torsion-free.
\end{proposition}

\begin{proof}
We first assume that $i=t$. In this case, it follows from the definition that 
$$I=(x_1x_{t+1}, x_1x_{t+2}, \ldots, x_1x_{n}, x_2x_{t+2}, \ldots, x_2x_n, \ldots, x_tx_{2t}, x_tx_{2t+1}, \ldots, x_tx_n).$$ It is routine to check that $I$ is the edge ideal of a bipartite graph   with the vertex set $ \{1, 2, \ldots, t\} \cup \{t+1, t+2, \ldots,  n\}$.
In addition, \cite[Corollary 14.3.15]{V1} implies that $I$ is normally torsion-free. Now, let $i>t$. One can conclude from the definition that the minimal generators of $I$ are as follows:
$$x_i x_n, x_i x_{n-1},  \ldots, x_i x_{i+t}, x_{i-1} x_n, x_{i-1} x_{n-1},  \ldots, x_{i-1}x_{i+t-1}, $$
$$ x_{i-2} x_n, x_{i-2} x_{n-1},  \ldots, x_{i-2} x_{i+t-2},  \ldots, x_1x_n, x_1 x_{n-1},  \ldots, x_1 x_{t+1}.$$
Our strategy  is to use Corollary \ref{Cor. 1}. To do  this, one has to show that  $I(\mathfrak{m}\setminus \{x_z\})$ is   normally torsion-free  for all $z=1, \ldots, n$, where $\mathfrak{m}=(x_1, \ldots, x_n)$. First, let  $1\leq z \leq i$.  
Direct computation gives that 
\begin{align*}
I(\mathfrak{m}\setminus \{x_z\})&= (x_{\alpha} x_{\beta}~:~ \alpha = 1, \ldots, z-1, \beta = t+1, \ldots, z+t-1, \beta - \alpha \geq t) \\
& + (x_n, x_{n-1}, \ldots, x_{z+t}). 
\end{align*}
One can  easily see that  
$(x_{\alpha} x_{\beta}~:~ \alpha = 1, \ldots, z-1, \beta = t+1, \ldots, z+t-1, \beta - \alpha \geq t)$ is the edge ideal of a bipartite graph with the following vertex set 
$$\{1, \ldots, z-1\} \cup \{t+1, \ldots, z+t-1\},$$  
and so is normally torsion-free. Also, we know that every prime monomial ideal is normally torsion-free. Moreover, on account of the two ideals 
$(x_n, x_{n-1}, \ldots, x_{z+t})$ and 
$(x_{\alpha} x_{\beta}~:~ \alpha = 1, \ldots, z-1, \beta = t+1, \ldots, z+t-1, \beta - \alpha \geq t)$ do not have any common variables, we conclude from \cite[Theorem 2.5]{SN} that 
$I(\mathfrak{m}\setminus \{x_z\})$ is   normally torsion-free as well. Now, let $i+1 \leq z \leq n$. It is routine to check that 
\begin{align*}
I(\mathfrak{m}\setminus \{x_z\})&= (x_{\alpha} x_{\beta}~:~ \alpha = z-t+1, \ldots, i, \beta = z+1, \ldots, n, \beta - \alpha \geq t) \\
& + (x_1,  x_2,  \ldots, x_{z-t}). 
\end{align*}
It is not hard to see that  $(x_{\alpha} x_{\beta}~:~ \alpha = z-t+1, \ldots, i, \beta = z+1, \ldots, n, \beta - \alpha \geq t)$ 
is the edge ideal of a bipartite graph with the following vertex set $$\{z-t+1, \ldots, i\} \cup \{z+1, \ldots, n\},$$ 
and hence is normally torsion-free. In addition, by a similar argument, the ideal $(x_1,  x_2,  \ldots, x_{z-t})$ is normally torsion-free. Furthermore, since $(x_1,  x_2,  \ldots, x_{z-t})$ and 
$(x_{\alpha} x_{\beta}~:~ \alpha = z-t+1, \ldots, i, \beta = z+1, \ldots, n, \beta - \alpha \geq t)$ have no common variables, we are able to derive from \cite[Theorem 2.5]{SN} that 
$I(\mathfrak{m}\setminus \{x_z\})$ is   normally torsion-free too. Inspired by Corollary \ref{Cor. 1}, we conclude that $I$ is nearly normally torsion-free, as claimed. 
\end{proof}


We illustrate the statement of Proposition~\ref{App. 2} through the following example. 

\begin{example}
Let $u= x_4x_7$ and $t=3$. If $v=x_ax_b \in \mathcal{G}(B_3(u))$, then $a \in [1,4]$ and $b \in [4,7]$. The complete list of minimal generators of $B_3(x_4x_{7})$ is given below:
\begin{center}
\begin{tabular}{ c c c c}
$x_1x_4$\\
$x_1x_5$\quad & \quad $x_2x_5$\\
$x_1x_6$\quad & \quad $x_2x_6$\quad & \quad $x_3x_6$\\
$x_1x_7$\quad & \quad $x_2x_7$\quad & \quad $x_3x_7$\quad & \quad $x_4x_7$
\end{tabular}
\end{center}
The monomial localizations of $I$ at $\mathfrak{m}\setminus \{x_k\}$, for each $k=1, \ldots 7$, are listed below:

\begin{center}
\begin{tabular}{l}
$I(\mathfrak{m}\setminus \{x_1\})=(x_4,x_5,x_6,x_7)$;\\
$I(\mathfrak{m}\setminus \{x_2\})=(x_5,x_6,x_7)+(x_1x_4)$;\\
$I(\mathfrak{m}\setminus \{x_3\})=(x_6,x_7)+(x_1x_4,x_1x_5,x_2x_5) $;\\
$I(\mathfrak{m}\setminus \{x_4\})=(x_7)+(x_1,x_2x_5,x_2x_6,x_3x_6)$;\\
$I(\mathfrak{m}\setminus \{x_5\})=(x_1,x_2)+(x_3x_6,x_3x_7,x_4x_7)$;\\
$I(\mathfrak{m}\setminus \{x_6\})=(x_1,x_2,x_3)+(x_4x_7)$;\\
$I(\mathfrak{m}\setminus \{x_7\})=(x_1,x_2,x_3,x_4)$. 
\end{tabular}
\end{center}

Therefore, for each $k=1, \ldots, 7$, the monomial localization $I(\mathfrak{m}\setminus \{x_k\})$ is either a monomial prime ideal, or sum of a monomial prime ideal and an edge ideal of a bipartite graph, which have no common variables. In each of the above cases, $I(\mathfrak{m}\setminus \{x_k\})$ is a normally torsion-free ideal. Hence, it follows from Corollary~\ref{Cor. 1} that $I$ is nearly normally torsion-free.
\end{example}

We further investigate  the class of ideals of type $B_t(u)$ in the context of normally torsion-freeness  and nearly normally torsion-freeness. Given $a,b \in \mathbb{N}$, we set $[a,b]:=\{c \in \mathbb{N}: a \leq c \leq b\}$. Let $u:=x_{i_1}x_{i_2}\cdots x_{i_{d-1}}x_{i_d}$ and 
\begin{equation}\label{ak}
A_k:=[(k-1)t+1,i_k], \text{ for each } k=1, \ldots d.
\end{equation}

We can describe the minimal generators of $I=B_t(u)$ in the following way: if $x_{j_1}\cdots x_{j_d} \in \mathcal{G}(I)$, then for each $k=1, \ldots ,d$, we have $j_k \in A_k$.  For example, a complete list of minimal generators of $B_2(x_3x_5x_{7})$ is given below:
\begin{center}
\begin{tabular}{ c c c }
$x_1x_3x_5$\\
$x_1x_3x_6$\\
$x_1x_3x_7$\\

$x_1x_4x_6$\quad & \quad $x_2x_4x_6$\\
$x_1x_4x_7$\quad & \quad $x_2x_4x_7$\\
$x_1x_5x_7$\quad & \quad $x_2x_5x_7$\quad & \quad $x_3x_5x_7$
\end{tabular}
\end{center}
Here, $A_1=[1,3]$, $A_2=[3,5]$, and $A_3=[5,7]$.

Let $u=x_{i_1}x_{i_2}\cdots  x_{i_{d-1}}x_{i_d}$ with $i_{d-1} \leq(d-1)t$, then it forces us to have $i_k \leq kt$, for each $k=1,\ldots, d-2$ because $u$ is a $t$-spread monomial. In this case, $A_i\cap A_j = \emptyset$. For example, a complete list of minimal generators of $B_3(x_3x_6x_9)$ is given below:
\begin{center}
\begin{tabular}{ c c c }
$x_1x_4x_7$\\
$x_1x_4x_8$\\
$x_1x_4x_9$\\
$x_1x_5x_8$\quad & \quad $x_2x_5x_8$\\
$x_1x_5x_9$\quad & \quad $x_2x_5x_9$\\
$x_1x_6x_9$\quad & \quad $x_2x_6x_9$\quad & \quad $x_3x_6x_9$

\end{tabular}
\end{center}

Here, $A_1=[1,3]$, $A_2=[4,6]$, and $A_3=[7,9]$. 

\begin{remark}\label{dpartiteduniform}
If $u=x_{i_1}x_{i_2}\cdots  x_{i_{d-1}}x_{i_d}$ with $i_{d-1} \leq(d-1)t$, then $B_t(u)$ is an edge ideal of a {\em  $d$-uniform $d$-partite} hypergraph whose vertex partition is given by $A_1, \ldots, A_d$.   Recall that $\mathcal{H}$ is a  $d$-partite hypergraph if its vertex set $V_{\mathcal{H}}$ is a disjoint union of sets $V_1, \ldots, V_d$ such that if $E$ is an edge of $\mathcal{H}$, then $|E \cap V_i | \leq 1$. Moreover, a hypergraph is $d$-uniform if each edge of $\mathcal{H}$ has size $d$.  In particular, if $\mathcal{H}$ is a $d$-uniform $d$-partite hypergraph with vertex partition $V_1, \ldots, V_d$, then $|E|=d$ and $|E \cap V_i | =1$ for each $E \in E_{\mathcal{H}}$.
\end{remark}

Let $\mathcal{A}_t$ denote the family of all ideals of the form $B_t(u)$ such that  $B_t(u)$ is an edge ideal of a $d$-partite $d$-uniform hypergraph, for some $d$. From the above discussion it follows that $B_t(u) \in \mathcal{A}_t$ if and only if there exists some $d$ such that  $u=x_{i_1}x_{i_2}\cdots  x_{i_{d-1}}x_{i_d}$ with $i_{d-1} \leq(d-1)t$.

Let $\mathcal{H}$ be a hypergraph. A sequence $v_1, E_1, v_2, E_2, \ldots, v_s, E_s, v_{s+1}=v_1$ of distinct edges  and vertices of $\mathcal{H}$ is called a cycle in $\mathcal{H}$ if $v_i,v_{i+1} \in E_i$, for all $i=1, \ldots, s$. Such a cycle is called {\em special} if no edge contain more than two vertices of the cycle. After translating the language of simplicial comlexes to hypergraphs, it can be seen from \cite[Theorem 10.3.16]{HH1} that if a hypergraph does not have any special odd cycles, then its edge ideal is normally torsion-free. The following example shows that hypergraphs with edge ideals in $\mathcal{A}_t$, may contain special odd cycles. 

\begin{example}\label{oddcycle}
Let $\mathcal{H}$ be the hypergraph whose edge ideal is $I=B_3(x_3x_6x_9x_{12})$. Then the following sequence of the vertices and the edges of $\mathcal{H}$ gives a special odd cycle. 
\[
1, \{1,4,9,12\}, 9, \{2,5,9,12\}, 5, \{1,5,8,11\}, 1
\]
\end{example}

For any monomial ideal $I \subset R=K[x_1, \ldots,x_n]$, the  {\it deletion} of $I$ at $x_i$ with $1\leq i \leq n$, denoted by $I\setminus x_i$, is obtained by setting $x_i=0$ in every minimal generator of $I$, that is, we delete every minimal generator $u\in \mathcal{G}(I)$ with $x_i\mid u$. For a monomial $u \in R$, we denote the support of $u$ by $\mathrm{supp}(u)=\{x_i : x_i|u \}$. Moreover, for a monomial ideal $I\subset R$, we set $\mathrm{supp}(I)=\cup_{u \in \mathcal{G}(I)} \mathrm{supp}(u)$. Then we observe the following:

\begin{remark}\label{rem1}
Let $u=x_{i_1}x_{i_2}\cdots  x_{i_{d-1}}x_{i_d}$ with $i_{d-1} \leq(d-1)t$, and $I=B_t(u) \subset R$. Note that we have $I \in \mathcal{A}_t$. 
\begin{enumerate}
\item If $i_k < kt$, for some $1 \leq k \leq d-1$, then the variables $x_{i_k+1}, \ldots, x_{kt}$ do not appear in $\mathrm{supp}(I)$. For example, a complete list of minimal generators of $B_3(x_2x_5x_9)$ is given below:
\begin{center}
\begin{tabular}{ c c c }
$x_1x_4x_7$\\
$x_1x_4x_8$\\
$x_1x_4x_9$\\
$x_1x_5x_8$\quad & \quad $x_2x_5x_8$\\
$x_1x_5x_9$\quad & \quad $x_2x_5x_9$\\
\end{tabular}
\end{center}
Here, $A_1=[1,2]$, $A_2=[4,5]$, and $A_3=[7,9]$, and $x_3, x_6 \notin \mathrm{supp}(B_3(x_2x_5x_9))$.

\item If $i_d<n$, then $x_{{i_d}+1}, \ldots, x_n \notin \mathrm{supp} (I)$. Hence we can always assume that $i_d=n$. 

\item If $i_k > (k-1)t+1$, that is, $|A_k| >1$  for some $k = 1, \ldots, d$, then $I\setminus x_{i_k}=B_t(v)$ where $v$ is chosen with the following property: $v=x_{j_1}\ldots x_{j_d} \in \mathcal{G}(I\setminus x_{i_k})$ and for any other $w=x_{l_1}\ldots x_{l_d}  \in \mathcal{G}(I\setminus x_{i_k})$, we have $l_k \leq j_k$. For example, $B_3(x_2x_5x_9)\setminus x_5=B_3(x_1x_4x_9)$. Therefore,  we conclude that $I\setminus x_{i_k} \in \mathcal{A}_t$, for all $x_{i_k}$, with $k = 1, \ldots, d$.

\item The definition of $A_k$ immediately implies that $|A_k|=1$ if and only if $i_k=(k-1)t+1$. Moreover, if $i_k=(k-1)t+1$ for some $1 \leq k \leq d$, then $i_j=(j-1)t+1$, for all $j \leq k$ because $u$ is a $t$-spread monomial. In this case $x_1x_{t+1}\cdots x_{(k-1)t+1}$ divides every minimal generator of $B_t(u)$ and hence $B_t(u)=x_1x_{t+1}\cdots x_{(k-1)t+1} J$, where $J$ can be identified as a $t$-spread principal Borel ideal generated in degree $d-k$ in it's ambient polynomial ring. Hence, $J=B_t(v) \in \mathcal{A}_t$, with $v=u/x_1x_{t+1}\cdots x_{(k-1)t+1} $.

For example, if $u=x_1x_4x_9x_{12}$, then $B_3(u)=x_1x_4J$ with $J=B_t(v) \subset K[x_7, \ldots, x_{12}]$ and $v=x_9x_{12}$. In fact, by substituting $y_{i-6}=x_i$ for $i=7, \ldots, 12$,  $J$ can be identified with $B_3(y_3y_6)\subset K[y_1, \ldots, y_6]$. 

\item If $|A_k| \geq 2$, then $i_k > (k-1)t+1$. This forces  $|A_j| \geq 2$ for all $j=k, \ldots,d$ because $u$ is a $t$-spread monomial. 

\end{enumerate}
\end{remark}

In what follows, our aim is to show that for any fixed $t$, all ideals in $A_t$ are normally torsion-free. It is known from \cite[Corollary]{AEL} that $B_t(u)$ satisfies the persistence property and the Rees algebra $\mathcal{R}(B_t(u))$ is a normal Cohen-Macaulay domain. It is a well-known fact that for any non-zero graded ideal $I \subset R=K[x_1, \ldots, x_n]$, if $\mathcal{R}(I)$ is Cohen-Macaulay then $\lim_{k \rightarrow \infty} \mathrm{depth}(R/I^k)= n-\ell(I)$, for example see \cite[Proposition 10.3.2]{HH1}, where $\ell(I)$ denotes the analytic spread of $I$, that is, the Krull dimension of the fiber ring $\mathcal{R}(I)/\mathfrak{m}\mathcal{R}(I)$. This leads us to the following corollary which will be used in the subsequent results.  Note that due to Remark~\ref{rem1}(1), one needs to pay attention to the ambient ring of $B_t(u)$. Here by the ambient ring, we mean the polynomial ring $R$ containing $B_t(u)$ such that all variables in $R$ appear in $\mathrm{supp}(B_t(u))$.

\begin{corollary}\label{cor1}
Let $I=B_t(u)\subset R=K[x_i: x_i \in \mathrm{supp}(I)]$. If $\ell(I)<n$, then $\lim_{k \rightarrow \infty} \mathrm{depth}(R/I^k)\neq 0$. In particular, if $\ell(I)<n$, then $\mathfrak{m} \notin \mathrm{Ass} (R/I^k)$, where $\mathfrak{m}$ is the unique graded maximal ideal of $R$.
\end{corollary}


Next we compute $\ell(I)$, for all $I \in \mathcal{A}_t$. For this, we first recall  the definition of  linear relation graph from \cite[Definition 3.1]{HQ}. 

\begin{definition}\label{linearrelationgraphdef}
Let $I\subset R$ be a monomial ideal with $\mathcal{G}(I)=\{u_1, \ldots, u_m\}$. The {\em linear relation graph} $\Gamma$ of $I$ is the graph with the edge set 
\[
E(\Gamma)=\{\{i,j\}: \text{there exist $u_k$,$u_l \in \mathcal{G}(I)$ such that $x_iu_k=x_ju_l$}\},
\]
 and the vertex set $V(\Gamma)=\bigcup_{\{i,j\}\in E(\Gamma)}\{i,j\}$. 
 \end{definition}
 
 It is known from \cite[Lemma 5.2]{DHQ} that if $I$ is a monomial ideal generated in degree $d$ and the first syzygy of $I$ is generated in degree $d+1$ then  
\begin{equation}\label{eq1}
\ell(I)=r-s+1
\end{equation}
where $r$ is the number of vertices and $s$ is the number of connected components of the linear relation graph of $I$. If $u$ is a $t$-spread monomial of degree $d$, then it can be concluded from \cite[Theorem 2.3]{AEL} that the first syzygy of $B_t(u)$ is generated in degree $d+1$. Hence  \cite[Lemma 5.2]{DHQ}  gives a way to compute the analytic spread of $B_t(u)$. Before proving the other main result of this section, we first analyze the linear relation graph of $B_t(u)$.

\begin{lemma}\label{aklemma}
Let $u=x_{i_1}x_{i_2}\cdots x_{i_{d-1}}x_{i_d}$ and $A_k=[(k-1)t+1,i_k]$, for each $k=1, \ldots d$. Moreover, let $\Gamma$ be the linear relation graph of $I=B_t(u)$. Then we have the following:

\begin{enumerate}
\item[(i)]  For some $k=1, \ldots, d$, $|A_k| \geq 2$ if and only if $A_k \subseteq V(\Gamma)$. Moreover, if $|A_k| \geq 2$, then the induced subgraph of $\Gamma$ on $A_k$ is a complete graph.
\item[(ii)] Let $i_k < kt+1$, for some $1 \leq k \leq d$. 
Then $\Gamma$ does not contain any edge with one endpoint in $A_r$ and the other endpoint in $A_s$ for any $1 \leq r<s\leq k$. 
\end{enumerate}
\end{lemma}

\begin{proof}
(i)  First we show that if $A_k \subseteq V(\Gamma)$, then $|A_k| \geq 2$. Assume that $|A_k|=1$, for some $k$.  Then from Remark~\ref{rem1}(4), we see that $i_k=(k-1)t+1$. In this case, each generator of $B_t(u)$ is a multiple of $x_{i_k}$ and by following the definition of $\Gamma$, we see that $A_k=\{i_k\} \not\subseteq V(\Gamma) $. 

Conversely,  we show that if $|A_k| \geq 2$, for some $k$, then $A_k \subseteq V(\Gamma)$ and the induced subgraph on $A_k$ is a complete graph. Take any $f,h \in A_k$ with $f<h$. If $k=1$, then  $f<h \leq i_1$ and the monomials $v=x_h(u/x_{i_1})$ and $v'=x_f(u/x_{i_1})$ belong to $\mathcal{G}(I)$ because $I$ is $t$-spread strongly stable. Hence $\{f,h\} \in \Gamma$, as required. Similar argument shows that if $|A_d| \geq 2$, then  $|A_d| \subseteq V(\Gamma)$ and the induced subgraph on $|A_d|$ is also a complete graph.  Now assume that $1 < k < d$. Since $|A_k| \geq 2$, we have $i_k > (k-1)t+1$ and $(k-1)t+1 \leq f<h \leq i_k$. By using the fact that $I$ is $t$-spread strongly stable, we see that $v=x_1 \cdots x_{(k-2)t+1}x_f x_{kt+1}\cdots x_{i_d} \in \mathcal{G}(I)$ and $v'=x_1 \cdots x_{(k-2)t+1}x_h x_{kt+1}\cdots x_{i_d} \in \mathcal{G}(I)$. Moreover, $v=(v'/x_h)x_f$ and $\{f,h\} \in E(\Gamma)$, as required.

(ii) If $|A_r|=1$ or $|A_s|=1$, then by using (i), we see that the statement holds trivially.  Assume that $|A_r| \geq 2$ and $|A_s| \geq 2$. Since $i_k< kt+1$, we have $i_j<jt+1$, for all $j=1, \ldots, k$ because $u$ is a $t$-spread monomial. In this case $A_r \cap A_s = \emptyset$ for all $1 \leq r < s \leq k$. Moreover, $|A_r| <t$ for all $r=1, \ldots, k$. Take $f \in A_r$ and $h \in A_s$ for some $1<r < s \leq k$. Then for any $u \in \mathcal{G}(I)$ with $x_h | u$, we have $(u/x_h)x_f \notin G(I)$. This can be easily verified because $(u/x_h)x_f$ is not a $t$-spread monomial as it contains two variables with indices in $A_r$ and $|A_r|<t$. Hence, we do not have any edge in $\Gamma$ of the form $\{f,h\}$ where $f \in A_r$ and $h \in A_s$.
\end{proof}

\begin{proposition}\label{limdepth}
Let  $I=B_t(x_{i_1}x_{i_2}\cdots  x_{i_{d-1}}x_{i_d}) \subset R=K[x_i:  x_i \in \mathrm{supp}(I)]$ with $i_{d-1} \leq(d-1)t$. Then $\ell(I) < n=|\mathrm{supp}(I)|$. In particular, $\mathfrak{m} \notin \mathrm{Ass} (R/I^k)$ for all $k\geq 1$, where $\mathfrak{m}$ is the unique graded maximal ideal of $R$.
\end{proposition}

\begin{proof}
If $I$ is a principal ideal, then the assertion holds trivially. Therefore, we may assume that $I$ is not a principal ideal. To show that $\ell(I)<n$, from the equality in (\ref{eq1}), it is enough to prove that the linear relation graph $\Gamma$ of $I$ has more than one connected components. Let $u=x_{i_1}x_{i_2}\cdots x_{i_{d-1}}x_{i_d}$ and $A_k=[(k-1)t+1,i_k]$, for $k=1, \ldots d$. Since $I$ is not principal, it follows from Remark~\ref{rem1}(4) that $|A_k| \geq 2$, for some $1 \leq k \leq d$.  

Since $i_{d-1} \leq(d-1)t$, one can deduce from Lemma~\ref{aklemma} that we do not have any edge in $\Gamma$ of the form $\{f,h\}$ with $f$ and $h$ in different $A_k$'s and if $|A_k| \geq 2$, for some $i$, then $A_k \subset V(\Gamma)$ and the induced subgraph on $A_k$ is a complete graph.  This shows that $V(\Gamma)$ is the union of all $A_i$'s for which $|A_i|\geq 2$. Moreover, each such $A_i$ determines a connected component in $\Gamma$. Hence $\Gamma$ has only one connected components if and only if $|A_d|\geq 2$ and $|A_i| =1$, for all $i=1, \ldots, d-1$. In this case, $\ell(I)= |A_d|<n$. Otherwise, $\Gamma$ has at least two connected components and again we obtain $\ell(I)<n$, as required. Then the assertion $\mathfrak{m} \notin \mathrm{Ass} (R/I^k)$, for all $k \geq 1$, follows from Corollary~\ref{cor1}.
\end{proof}

Before proving Theorem~\ref{main-NTF}, we first recall the following theorem which gives a criterion to check whether a square-free monomial ideal is normally torsion-free or not.

\begin{theorem}\label{use}\cite[Theorem 3.7]{SNQ}
Let $I$ be a square-free  monomial ideal in a polynomial ring $R=K[x_1, \ldots, x_n]$ over a field $K$ and $\mathfrak{m}=(x_1, \ldots, x_n)$. If there exists a square-free monomial  $v \in I$ such that $v\in \mathfrak{p}\setminus \mathfrak{p}^2$ for any $\mathfrak{p}\in \mathrm{Min}(I)$, and $\mathfrak{m}\setminus x_i \notin \mathrm{Ass}(R/(I\setminus x_i)^s)$ for all $s$ and $x_i \in \mathrm{supp}(v)$,  then the following statements hold:
\begin{itemize}
\item[(i)] $I$  is normally torsion-free. 
\item[(ii)]  $I$ is normal.
\item[(iii)]  $I$ has the strong persistence proeprty. 
\item[(iv)]  $I$ has the persistence property. 
\item[(v)]  $I$ has the symbolic strong persistence property. 
\end{itemize}

\end{theorem}

Now, we state the second main result of this section. 
\begin{theorem}\label{main-NTF}
Let  $I=B_t(x_{i_1}x_{i_2}\cdots  x_{i_{d-1}}x_{i_d}) \subset R=K[x_i:  x_i \in \mathrm{supp}(I)]$ with $i_{d-1} \leq(d-1)t$. Then $I$ is normally torsion-free.
\end{theorem}

\begin{proof}
We may assume that $i_k \neq (k-1)t+1$, for all $1 \leq k \leq d$. Otherwise, from Remark~\ref{rem1}(4), it follows that $I=wB_t(v)$ where $w$ is the product of all variables for which $i_k =(k-1)t+1$, $v=u/w$, and $B_t(v)$ is a $t$-spread principal Borel ideal in it's ambient ring. Then from \cite[Lemma 3.12]{SN}, it follows that $I$ is normally torsion-free if and only if $B_t(v)$ is normally torsion-free. Therefore, one may reduce the discussion to $B_t(v)$ whose generators are not a multiple of a fixed monomial. 
 
 Let $u=x_{i_1}x_{i_2}\cdots  x_{i_{d-1}}x_{i_d}$ with $i_1>1$. To show that $I=B_t(u)$ is normally torsion-free, we will use Theorem~\ref{use}. It can be seen from \cite[Theorem 4.2]{Claudia} that $u \in \mathfrak{p}\setminus \mathfrak{p}^2$, for all $\mathfrak{p} \in \mathrm{Min}(I)$. Recall that the family of ideals $\mathcal{A}_t$ is defined by: $B_t(u) \in \mathcal{A}_t$ if and only if there exists some $d$ such that  $u=x_{i_1}x_{i_2}\cdots  x_{i_{d-1}}x_{i_d}$ with $i_{d-1} \leq(d-1)t$. From Remark~\ref{rem1}(3), it follows that $I\setminus{x_{i_k}} \in \mathcal{A}_t$, for all $k=1, \ldots, d$. Furthermore, Proposition~\ref{limdepth} implies that the unique graded maximal ideal of the ambient ring of $I\setminus{x_{i_k}}$ does not belong to $\mathrm{Ass}(R/(I\setminus x_i)^s)$ for all $s$.
\end{proof}

Next, we prove the converse of Theorem~\ref{main-NTF} to obtain the complete characterization of normally torsion-free $t$-spread principal Borel ideals.

\begin{theorem}\label{complete}
Let  $I=B_t(x_{i_1}x_{i_2}\ldots  x_{i_{d-1}}x_{i_d}) \subset R=K[x_i:  x_i \in \mathrm{supp}(I)]$. Then  $I$ is normally torsion-free if and only if  $i_{d-1} \leq(d-1)t$. In other words, $I$ is normally torsion-free if and only if $I$ can be viewed as an edge ideal of a $d$-uniform $d$-partite hypergraph. 
\end{theorem}
\begin{proof}
Following Theorem~\ref{main-NTF}, it is enough to show that if $i_{d-1} >(d-1)t$, then $I$ is not normally torsion-free. Let $k$ be the smallest integer for which $i_k > kt$. As we explained in Remark~\ref{rem1}(5), if $i_k > kt$, then $i_j> jt$ and $|A_j| \geq 2$, for all $j= k, \ldots, d$. The sets $A_j$'s are defined in (\ref{ak}). It follows from Lemm~\ref{aklemma} that $A=A_k \cup A_{k+1} \cup \cdots  \cup A_d \subset V(\Gamma) $, where $\Gamma$ is linear relation graph of $I$ and for each $j=k, \ldots, d$,  the induced subgraph of $\Gamma$ on $A_j$ is a complete graph. Moreover, $i_j> jt$, for each $j=k, \ldots, d$ gives that $i_j \in A_j \cap A_{j+1} \neq \emptyset$, for each $j=k, \ldots, d-1$. Therefore, we conclude that the induced subgraph on $A$ is connected.

Set $\mathfrak{p}=(x_{(k-1)t+1}, \ldots, x_{i_d})$. Here the crucial observation is that if we take the monomial localization of $I$ at $\mathfrak{p}$; in other words, if we map all variables $x_i $ to 1 where $i \in A_1 \cup \cdots \cup A_{k-1}$, then we are reducing the degree of each generator of $B_t(u)$ by $k-1$. It is because $A \cap B = \emptyset$, where $B=A_1 \cup \cdots \cup A_{k-1}$. Hence $I(\mathfrak{p})$ can be viewed as a $t$-spread principal Borel ideal by a shift of indices of variables. More precisely, each  $j \in A $ is shifted to $j-(k-1)t$. Therefore, $I(\mathfrak{p})= B_t(x_{j_1} \cdots x_{j_k})$, where $t<j_1 < \cdots < j_k$. The linear relation graph of $I(\mathfrak{p})$ is isomorphic to the induced subgraph of $\Gamma$ on vertex set $A$. Then $\ell( I(\mathfrak{p})) = |\mathrm{supp} (B_t(x_{j_1} \cdots x_{j_k}))|=\mathrm{dim}(R(\mathfrak{p}))$, and $\lim_{k \rightarrow \infty} \mathrm{depth}(R(\mathfrak{p})/I(\mathfrak{p})^k)= 0$. This shows that $\mathfrak{p} \in \mathrm{Ass }(R/I^k)$, for some $k>1$. Moreover, $\mathfrak{p} \notin \mathrm{Ass}(R/I)$ because of \cite[Theorem 4.2]{Claudia}. Hence we conclude that $I$ is not normally torsion-free.
\end{proof}
In the subsequent example, we illustrate the construction of $\mathfrak{p}$ as in the proof of Theorem~\ref{complete}.
\begin{example}
Let $u=x_2x_7x_{10}x_{13}$ be a 3-spread monomial and $I=B_3(u)$. Here $i_1=2, i_2=7, i_3=10, i_4=13$. Moreover $i_1=2 < t=3$, but $i_2=7 >2.3$. Set $\mathfrak{p}=(x_4, x_5, \ldots, x_{13})$ as in the proof of Theorem~\ref{complete}. Then the minimal generators of $I(\mathfrak{p})$ are listed below. In the following table, $u \rightarrow v$ indicates that $u \in \mathcal{G}(I(\mathfrak{p}))$ and $v$ is the monomial obtained by shifting  each $j \in [4,13]$ to $j-3$.
\begin{center}
\begin{tabular}{ l l l l}
$x_4x_7x_{10} \rightarrow x_1x_4x_{7}$\\
$x_4x_7x_{11} \rightarrow x_1x_4x_{8}$\\
$x_4x_7x_{12} \rightarrow x_1x_4x_{9}$\\
$x_4x_7x_{13} \rightarrow x_1x_4x_{10}$\\
$x_4x_8x_{11} \rightarrow x_1x_5x_{8}$&$x_5x_8x_{11} \rightarrow x_2x_5x_{8}$\\
$x_4x_8x_{12} \rightarrow x_1x_5x_{9}$&$x_5x_8x_{12}  \rightarrow x_2x_5x_{9}$\\
$x_4x_8x_{13} \rightarrow x_1x_5x_{10}$&$x_5x_8x_{13} \rightarrow x_2x_5x_{10}$\\
$x_4x_9x_{12} \rightarrow x_1x_6x_{9}$&$x_5x_9x_{12} \rightarrow x_2x_6x_{9}$&$x_6x_9x_{12} \rightarrow x_3x_6x_{9}$\\
$x_4x_9x_{13} \rightarrow x_1x_6x_{10}$&$x_5x_9x_{13} \rightarrow x_2x_6x_{10}$&$x_6x_9x_{13}\rightarrow x_3x_6x_{10}$\\
$x_4x_{10}x_{13} \rightarrow x_1x_7x_{10}$&$x_5x_{10}x_{13} \rightarrow x_2x_7x_{10}$&$x_6x_{10}x_{13}\rightarrow x_3x_7x_{10}$&$x_7x_{10}x_{13}\rightarrow x_4x_7x_{10}$
\end{tabular}
\end{center}
Therefore, $I(\mathfrak{p})$ can be viewed as a $3$-spread principal Borel ideal $B_3(x_4x_7x_{10})$. A direct computation in Macaulay2 \cite{GS} shows that $(x_1, \ldots, x_{10})$ is an associated prime of the third power of $B_3(x_4x_7x_{10})$. Consequently, $\mathfrak{p}$ is also an associated prime of the third power of $I$. 
\end{example}

When $u$ is a monomial of degree 3, then we have the following: 

\begin{lemma}
Let $u=x_ax_bx_n$ be a $t$-spread monomial and $I=B_t(u)\subset R=K[x_i:  x_i \in \mathrm{supp}(I)]$. Then we have the following:
\begin{enumerate}
\item[\em{(i)}] if $b < 2t+1$, then $I$ is normally torsion-free.

\item[\em{(ii)}] $a=1$ and $b\geq2t+1$, then $I$ is nearly normally torsion-free.

\item[\em{(iii)}] $a>1$ and $b\geq2t+1$, then $I$ is not nearly normally torsion-free.

\end{enumerate}
\end{lemma}

\begin{proof}
(i) It follows from Theorem~\ref{main-NTF}.

(ii) Let $u\in \mathcal{G}(I)$. Then $u=x_1x_{i_2}x_{i_3}$, where $i_2\in \{t+1, \ldots, b\}$ and $i_3\in \{2t+1,\ldots,n\}$. It can be easily seen that $I=x_1J$, where $J=B_t(x_{b}x_n)$. It follows from Proposition~\ref{App. 2} that $B_t(x_{b}x_n)$ is nearly normally torsion-free. Then we get the required result by using Lemma~\ref{Lem. Multipe}.  

(iii) We will prove the assertion by constructing two monomial prime ideals $\mathfrak{p}_1$ and $\mathfrak{p}_2$ that belong to $\mathrm{Ass}(R/I^k)$, for some $k>1$, but $\mathfrak{p}_1, \mathfrak{p}_2 \notin \mathrm{Ass}(R/I)$. Recall from (\ref{ak}) that if $v \in \mathcal{G}(I)$, then $v=x_{a'}x_{b'}x_{c'}$ with $ a' \in A_1=[1, a]$, $b' \in A_2=[t+1,b]$, and $c'\in A_3=[2t+1, n]$. Since $a>1$, we have $|A_1|> 1$ and since $b\geq 2t+1$, we have $|A_2|> 1$ and $|A_3|>1$.

Let $\mathfrak{p}_1=(x_{t+1}, \ldots, x_{i_d})$. As shown in the last part of the proof of Theorem~\ref{complete}, $\mathfrak{p}_1 \in \mathrm{Ass}(R/I^k)$, for some $k>1$, but $\mathfrak{p}_1 \notin \mathrm{Ass}(R/I)$. Let $\mathfrak{p}_2=(x_1, x_{t+2}, x_{t+3}, \ldots, x_n)$, then from \cite[Theorem 4.2]{Claudia}, we have $\mathfrak{p}_2 \notin \mathrm{Ass}(R/I)$. We claim that $I(\mathfrak{p}_2)$ can be viewed as the $t$ spread principal Borel ideal $B_t(x_rx_n)$ where $r>t$. Indeed, by substituting $x_{t+1}=1$, all $t$-spread monomials of the form $x_1x_{t+1}x_{c'}$, with $c' \in A_3$ are reduced to $x_1x_{c'}$. Therefore, the monomials $x_1x_{b'}x_c \in \mathcal{G}(I)$ with $b' \in [t+2,b]$ and $c' \in [2t+2,n]$ do not appear in $\mathcal{G}(I(\mathfrak{p}_2))$.

Since $a>1$, we have $2 \in A_1$. Moreover, if $v=x_{a'}x_{b'}x_{c'} \in \mathcal{G}(I)$ with $a'\geq 2$, then $b' \geq t+2$ and $c' \geq 2t+2$ because $v$ is a $t$-spread monomial .  

By substituting $x_{i}=1$, for all $i \in [2,a]$, all monomials of the form $x_i x_{b'}x_{c'}$, with $b' \in [t+2, b]$, $c' \in [2t+2,n]$ are reduced to $x_{b'}x_{c'}$. This shows that $I(\mathfrak{p}_2)$ is generated in degree 2 by $t$-spread monomials of the following form:
\begin{center}
\begin{tabular}{l}
$x_1x_{c'}$ with $c' \in A_3=[2t+1,n]$; \\
 $x_{b'}x_{c'}$ with $b'\in [t+2,b]$ and $c' \in [2t+2,n]$ and $c' - b'\geq t$.
\end{tabular}
\end{center}

Note that $\mathrm{supp}(I(\mathfrak{p}_2))=\{x_1, x_{t+2}, x_{t+3}, \ldots, x_n\}$. Moreover, if we shift the indices of variables in $\{x_{t+2}, x_{t+3}, \ldots, x_n\}$ by $k\rightarrow k-t$, then $I(\mathfrak{p}_2)$ can be viewed as a $t$-spread principal Borel ideal $B_t(x_{r}x_{s})$ where $s=n-t$ and $r=b-t>t$ because $b \geq 2t+1$. Moreover, the linear relation graph of $B_t(x_{r}x_{s})$ is a connected with $ |\mathrm{supp}(I(\mathfrak{p}_2))| $ vertices. Hence, $\lim_{k \rightarrow \infty} \mathrm{depth}(R(\mathfrak{p}_2)/I(\mathfrak{p}_2)^k)= 0$, and consequently $\mathfrak{p}_2 \in \mathrm{Ass}(R/I^k)$, for some $k >1$. 
\end{proof}
 

  \noindent{\bf Acknowledgments.}
 We   would like to thank Professor Adam Van Tuyl for his valuable comments  in preparation of Theorem \ref{NTF1}.  
In addition, the authors are deeply grateful to the  referee for careful reading of the manuscript and    valuable suggestions which led to significant  improvements in the quality of this paper.


\end{document}